\documentclass[a4paper]{article}

\usepackage{amsfonts, amsmath, amsthm, empheq, amssymb, bm, setspace, mathrsfs}
\usepackage[numbers,sort]{natbib}
\usepackage[colorlinks,pdfborder=001,linkcolor=blue,citecolor=blue,backref=page]{hyperref}
\makeatletter

\newcommand{\Rmnum}[1]{\expandafter\@slowromancap\romannumeral #1@}
\makeatother

\allowdisplaybreaks[4]

\newtheorem{thm}{Theorem}[section]

\newtheorem{lem}[thm]{Lemma}

\def\R{\Bbb R}
\def\De{\Delta}
\def\al{\alpha}
\def\be{\beta}
\def\de{\delta}
\def\ep{\epsilon}

\def\e{\mathrm e}

\def\f{\frac}

\def\ga{\gamma}
\def\Ga{\Gamma}

\def\la{\lambda}
\def\vp{\varphi}
\def\na{\nabla}
\def\om{\omega}
\def\Om{\Omega}
\def\ov{\overline}
\def\pa{\partial}

\def\wt{\widetilde}

\title{\Large\bf\boldmath
Inverse coefficients problem for a magnetohydrodynamics system}

\author{\large Xinchi HUANG$^\dag$}

\date{October 2016}

\begin{document}
\maketitle

\renewcommand{\thefootnote}{\fnsymbol{footnote}}
\footnotetext{\hspace*{-5mm}
\begin{tabular}{@{}r@{}p{13cm}@{}}
$^\dag$
& Graduate School of Mathematical Sciences,
the University of Tokyo,
3-8-1 Komaba, Meguro-ku, Tokyo 153-8914, Japan.
E-mail: huangxc@ms.u-tokyo.ac.jp,
\end{tabular}}

\begin{abstract}
\noindent In this article, we consider a magnetohydrodynamics system for incompressible flow in a three-dimensional bounded domain. Firstly, we give the stability results for our inverse coefficients problem. Secondly, we establish and prove two Carleman estimates both for direct problem and inverse problem. Finally, we complete the proof of stability result in terms of the above Carleman estimates.
\end{abstract}


\section{Introduction}
\label{sec-intro}
Magnetohydrodynamics(MHD) is the study of the magnetic properties of electrically conducting fluids such as plasmas, liquid metals and salt water. The set of equations in three dimension is introduced by combining the Navier-Stokes equations and Maxwell's equations:
$$
\left\{
\begin{aligned}
& \ \pa_t u + \mathrm{div}(\rho u\otimes u - P(u,p)) - \mu \mathrm{rot}\; H\times H = F, \\
& \ \pa_t H -\mathrm{rot}(u\times H) = -\mathrm{rot}(\frac{1}{\sigma\mu}\mathrm{rot}\; H), \\
& \ \mathrm{div} u = \mathrm{div} H = 0
\end{aligned}
\right.
$$
where the notations $\times$ and $\otimes$ mean cross product and outer product which are defined as follows: for any vectors $A=(A_1,A_2,A_3)^T$ and $B=(B_1,B_2,B_3)^T$,
$$
\begin{aligned}
&A\times B:= (A_2B_3 - A_3B_2, A_3B_1 - A_1B_3, A_1B_2 - A_2B_1), &A\otimes B:= A\;B^T.
\end{aligned}
$$
Here, $u=(u_1,u_2,u_3)^T$, $H=(H_1,H_2,H_3)^T$ denote the velocity vector and the magnetic field intensity respectively. $P(u,p)$ denotes the stress tensor which is determined by generalized Newton's law as
$$
P(u,p) = -pI + 2\nu\mathcal{E}(u)
$$
where $p$ denotes the pressure and $\mathcal{E}(u)$ is called Cauchy stress tensor defined by
$$
\mathcal{E}(u) := \frac{1}{2}(\na u + (\na u)^T).
$$
The coefficient $\nu$ is related to the viscosity of the fluids. Furthermore, $\sigma$ and $\mu$ are the electrical conductivity and magnetic permeability respectively. For the derivation of above equations, we refer to Li and Qin \cite{LQ13}. We don't pay attention to temperature distribution of the fluid and thus neglect the energy equation.
\vspace{0.2cm}

There are some papers for MHD systems. \cite{DL13,FL15} studied some regularity criteria for incompressible MHD system in three dimension. In \cite{DL13}, the authors established some general sufficient conditions for global regularity of strong solutions to incompressible three-dimensional MHD system. While \cite{FL15} gave a logarithmic criterion for generalized MHD system. We should also mention the study of exact controllability for MHD. Hav\^{a}rneanu, Popa and Sritharan \cite{HPS06,HPS07} studied it with locally internal controls both in two and in three dimension. In their papers, they have established a kind of Carleman estimate for MHD system in order to solve their controllability problems. However, it is not enough to consider inverse problems, especially inverse source problems. We will clarify this statement later.
\vspace{0.2cm}

In this article, our main method is Carleman estimate. It is an $L^2$- weighted estimate with large parameter(s) for a solution to a partial differential equation. The idea was first introduced by Carleman \cite{C39} for proving the unique continuation for a two-dimensional elliptic equation. From the 1980s, there have been great concerns for the estimate itself and its applications as well. For remarkable general treatments, we refer to \cite{E86,H85,I90,I98,T96,T81}. Carleman estimate has then become one of the general techniques in studying unique continuation and stability for inverse problems. Since then, there are many papers considering different inverse problems for a variety of partial differential equations. We list some work for the well-known equations in mathematical physics. For hyperbolic equation, Bellassoued and Yamamoto \cite{BY06} considered the inverse source problem for wave equation and give a stability inequality with observations on certain sub-boundary. Gaitan and Ouzzane \cite{GO13} proved a lipschitz stability for the inverse problem which reconstructs an absorption coefficient for a transport equation with also boundary measurements. For heat(parabolic) equation, Yamamoto \cite{Y09} have given a great survey by summarizing different types of Carleman estimates and methods for applications to some inverse problems (see also the references therein). Moreover, Choulli, Imanuvilov, Puel and Yamamoto \cite{CIPY13} has worked on the inverse source problem for linearized Navier-Stokes equations with data in arbitrary sub-domain.
\vspace{0.2cm}

To authors' best knowledge, there are few papers on Carleman estimates for MHD system. Recall that in \cite{HPS06,HPS07}, the authors have proved a Carleman estimate for the adjoint MHD system in order to prove the exact controllability. However, in their Carleman estimate, the observation of the first spatial derivative of external force $F$ is needed which makes it difficult to consider inverse source problems in general case and thus it is even not suitable for inverse coefficient problems. In this article, we intend to establish Carleman estimates for the above MHD system and then give the stability inequality for the principal coefficients.

By taking the difference of two states for MHD systems with different coefficients, it is enough to consider an inverse source problem for a linearized MHD system. The main difficulty lies in the first-order partial differential term in the source. We use the idea of \cite{Y09} in which the author dealt with a similar problem for equation of parabolic type by giving a Carleman estimate for a first-order partial differential operator. In this article, we modified the Carleman estimate for first-order partial differential operator in a vector-valued case. Then together with Carleman estimate for MHD system, we prove a Lipschitz stability for inverse coefficients problem and also a conditional stability of H\"{o}lder type under weaker assumptions. 
\vspace{0.2cm}

This article is organized as follows. In section 2, we introduce some notations and then give the concerned MHD system and precise statements for our inverse coefficients problem. In section 3, we establish Carleman inequalities both for direct problem and inverse problem. For direct problem, we need a Carleman estimate for MHD system. On the other hand, we prove the inequality for inverse problem in terms of a Carleman estimate for a first-order partial differential operator. In section 4, we complete the proof of the main results in section 2 by using the above Carleman inequalities.

\section{Notations and stability results}
\label{sec-main}

Let $\Om\subset\R^3$ be a bounded domain with smooth boundary. We set $Q:=\Om\times(0,T)$, $\Sigma :=\pa\Om\times(0,T)$.
In this article, we use the following notations. $\cdot^T$ denotes the transpose of matrices or vectors. Let $\pa_t = \f{\pa}{\pa t}, \; \pa_j = \f{\pa}{\pa x_j}, j = 1,2,3, \; \De = \sum_{j=1}^3 \pa_j^2, \; \na = (\pa_1, \pa_2, \pa_3)^T, \; \na_{x,t} = (\na,\pa_t)^T $

\begin{equation*}
(w \cdot \na) v = \left(\sum_{j=1}^3 w_j\pa_j v_1, \sum_{j=1}^3 w_j\pa_j v_2, \sum_{j=1}^3 w_j\pa_j v_3 \right)^T,
\end{equation*}
for $v = (v_1,v_2,v_3)^T$ and $w = (w_1,w_2,w_3)^T$. Henceforth let $n$ be the outward unit normal vector to $\pa\Om$ and let $\pa_n u:=\f{\pa u}{\pa n} = \na u \cdot n$. Moreover let $\ga = (\ga_1,\ga_2,\ga_3) \in (\Bbb N \cup \{0\})^3, \; \pa_x^{\ga} = \pa_1^{\ga_1} \pa_2^{\ga_2} \pa_3^{\ga_3}$ and $\vert \ga \vert = \ga_1 + \ga_2 +\ga_3$.

Furthermore, we introduce the following spaces:
\begin{equation*}
\left\{
\begin{aligned}
&W^{k,\infty}(D) :=\{w; \; \pa_t^{\gamma} w,\pa_x^{\gamma} w \in L^\infty (D), |\gamma|\le k \}, \ k\in \Bbb{N} \\
&H^{k,l}(D) :=\{w; \; \pa_t^{\gamma_0} w, \pa_x^{\gamma} w \in L^2(D), \; |\gamma_0|\le l, |\gamma|\le k \}, \ k,l\in \Bbb{N}\cup \{ 0 \}
\end{aligned}
\right.
\end{equation*}
for any sub-domain $D\subset Q$.
If there is no confusion, we also denote $(L^2(\Om))^3$ by $L^2(\Om)$, likewise $(H^{k,l}(D))^3$ by simply $H^{k,l}(D)$, $k,l\in \Bbb{N}\cup \{ 0 \}$.
\vspace{0.2cm}

In this article, we denote $\kappa = \sigma^{-1}$ the resistance. For simplicity, we just assume the magnetic permeability $\mu$ to be a constant(identically $1$). In fact, we consider the following MHD system:
\begin{equation}
\label{sy:MHD}
\left\{
\begin{aligned}
& \ \pa_t u - \mathrm{div}(2\nu\mathcal{E}(u)) + (u\cdot\nabla) u - (H\cdot\nabla) H + \na (p - \frac{1}{2} |H|^2)= 0, \\
& \ \pa_t H + \mathrm{rot}(\kappa\mathrm{rot}\; H) + (u\cdot\nabla) H - (H\cdot\nabla) u= 0, \\
& \ \mathrm{div}\; u = \mathrm{div}\; H = 0.
\end{aligned}
\right.
\end{equation}
Here, the viscosity $\nu=\nu(x)$ and the resistance $\kappa=\kappa(x)$ are time independent coefficients which admit a positive lower bound. Now we let $(u_i, p_i, H_i)$(i=1,2) are two sets of functions satisfying (\ref{sy:MHD}) corresponding to coefficients $(\nu_i, \kappa_i)$(i=1,2). That is,
\begin{equation}
\label{sy:MHD2}
\left\{
\begin{aligned}
&\pa_t u_i - \mathrm{div}(2\nu_i\mathcal{E}(u_i)) + (u_i\cdot\nabla) u_i - (H_i\cdot\nabla) H_i + \na p_i - \nabla H_i^T\!\cdot\! H_i = 0 &\quad in \ Q, \\
&\pa_t H_i + \mathrm{rot}(\kappa_i\mathrm{rot}\; H_i) + (u_i\cdot\nabla) H_i - (H_i\cdot\nabla) u_i = 0 &\quad in \ Q, \\
& \mathrm{div}\; u_i = 0, \quad \mathrm{div}\; H_i = 0 &\quad in \ Q.
\end{aligned}
\right.
\end{equation}
The sets of functions $(u_i, p_i, H_i, \nu_i,\kappa_i)$(i=1,2) are supposed to be smooth enough (e.g. $W^{2,\infty}(Q)$). Then we choose a function $d\in C^2(\ov{\Omega})$ such that
\begin{equation}
\label{con:d}
d>0 \ in \ \Om, \quad |\na d| > 0 \ on \ \ov{\Om}, \quad d=0 \ on \ \pa\Om\setminus\Gamma
\end{equation}
for any nonempty sub-boundary $\Gamma\subset\partial\Omega$. The existence of such function was proved in \cite{Y09}. In fact, we can choose a bounded domain $\Om_1$ with boundary smooth enough such that
\begin{equation}
\label{con:choice-Om1}
\Om \subsetneqq \Om_1, \quad \ov{\Gamma} = \ov{\pa\Om\cap\Om_1}, \quad \pa\Om\setminus\Gamma \subset \pa\Om_1,
\end{equation}
thus $\Om_1\setminus\ov{\Om}$ contains some non-empty open subset. It is a well-known result (see Imanuvilov, Puel and Yamamoto \cite{IPY09}, Fursikov and Imanuvilov \cite{FI96}) that there exists a function $\eta\in C^2(\ov{\Om})$ such that for any $\om\subset\subset \Om$,
$$
\eta |_{\pa\Om} = 0, \quad \eta >0 \ in \ \Om, \quad |\na \eta|>0, \ on \ \ov{\Om\setminus \om}.
$$
By choosing $\ov{\om}\subset \Om_1\setminus \ov{\Om}$ and applying the above result in $\Om_1$, we obtain our function $d$. Without special emphases, we use the function $d$ as above throughout this article.

Fix observation time $t_0\in (0,T)$. Before giving our stability result, we need furthermore the following two assumptions: 
\vspace{0.2cm}

(A1)$\quad \mathrm{det}\;\mathcal{E}(u_1(x,t_0)) \neq 0 \quad\quad for \ any\ x\in \ov{\Om}$, 

(A2)$\quad |\na d(x)\times\mathrm{rot}H_1(x,t_0)| \neq 0 \quad\quad for \ any\ x\in \ov{\Om}$.
\vspace{0.2cm}

\noindent Now we are ready to state our main result. $\Gamma\subset \pa\Om$ is an arbitrarily fixed relatively open sub-boundary. 
\begin{thm}
\label{thm:stability}
Under the assumptions (A1)-(A2) and the conditions
\begin{equation}
\label{con:bdy}
\nu_1(x) = \nu_2(x) \quad on\ \Gamma, \qquad \kappa_1(x) = \kappa_2(x),\ \na\kappa_1(x) = \na\kappa_2(x) \quad on\ \partial\Omega, 
\end{equation} 
there exists a constant $C>0$ such that
$$
\|\nu_1-\nu_2\|_{H^1(\Om)} + \|\kappa_1-\kappa_2\|_{H^1(\Om)} \le C\mathcal{D}
$$
for all $(u_i,p_i,H_i)\in H^{2,3}(Q)\times H^{1,2}(Q)\times H^{2,3}(Q)$ satisfying system (\ref{sy:MHD2}) for $i=1,2$.
\end{thm}
\noindent Here the measurement $\mathcal{D}$ denotes
\begin{align*}
&\mathcal{D}=  \|(u_1-u_2)(\cdot,t_0)\|_{H^2(\Om)} + \|(H_1-H_2)(\cdot,t_0)\|_{H^3(\Om)} + \|\na (p_1-p_2)(\cdot,t_0)\|_{L^2(\Om)} \\
&\hspace{0.8cm} +  \|u_1-u_2\|_{H^{0,2}(\Sigma)} + \|\na_{x,t} (u_1-u_2)\|_{H^{0,2}(\Sigma)} + \|p_1-p_2\|_{H^{\frac{1}{2},2}(\Sigma)} \\
&\hspace{0.8cm} + \|H_1-H_2\|_{H^{0,2}(\Sigma)} + \|\na_{x,t} (H_1-H_2)\|_{H^{0,2}(\Sigma)} .
\end{align*}
$H^{k,l}(\Sigma)\equiv H^{k}(0,T;H^{l}(\partial\Om))$($k,l\in\Bbb{N}$). The assumption (A1)-(A2) are strong because we need them to hold globally. Now consider the following weaker assumptions:
 \vspace{0.2cm}

(A1$^\prime$)$\quad \mathrm{det}\;\mathcal{E}(u_1(x,t_0)) \neq 0 \quad\quad for \ any\ x\in \ov{\Om_{3\ep}}$, 

(A2$^\prime$)$\quad |\na d(x)\times\mathrm{rot}H_1(x,t_0)| \neq 0 \quad\quad for \ any\ x\in \ov{\Om_{3\ep}}$
\vspace{0.2cm}

\noindent where $\Om_\ep := \{x\in \Om: d(x)>\ep\}$ for any $\ep>0$. Then we can derive a local stability result.
\begin{thm}
\label{thm:stabilityn}
Under the assumptions {\text(A1$^\prime$)-(A2$^\prime$)} and the conditions
\begin{equation}
\label{con:bdyn}
\nu_1(x) = \nu_2(x) \quad on\ \Gamma, \qquad \kappa_1(x) = \kappa_2(x),\ \na\kappa_1(x) = \na\kappa_2(x) \quad on\ \Gamma, 
\end{equation} 
there exist constants $C>0$ and $\theta\in (0,1)$ such that
\begin{align}
\label{eq:loc-stab}
\|\nu_1-\nu_2\|_{H^1(\Om_{5\ep})} + \|\kappa_1-\kappa_2\|_{H^1(\Om_{5\ep})} \le C(\mathcal{D} + M^{1-\theta}\mathcal{D}^{\theta})
\end{align}
for all $(u_i,p_i,H_i)\in H^{2,3}(Q)\times H^{1,2}(Q)\times H^{2,3}(Q)$ satisfying system (\ref{sy:MHD2}) for $i=1,2$.
\end{thm}
\noindent Here a prior bound $M$ and measurements $\mathcal{D}$ denote
\begin{align*}
&M = \sum_{j=0}^2\Big(\|\pa_t^j u\|_{H^{1,1}(Q)} + \|\pa_t^j H\|_{H^{1,0}(Q)} + \|\pa_t^j p\|_{L^2(Q)}\Big) + \|\nu\|_{H^1(\Om_{3\ep})} + \|\kappa\|_{H^1(\Om_{3\ep})}, \\
&\mathcal{D} =  \|(u_1-u_2)(\cdot,t_0)\|_{H^2(\Om_{3\ep})} + \|(H_1-H_2)(\cdot,t_0)\|_{H^3(\Om_{3\ep})} + \|\na (p_1-p_2)(\cdot,t_0)\|_{L^2(\Om_{3\ep})} \\
&\hspace{0.8cm} +  \|u_1-u_2\|_{H^{0,2}(\Gamma\times(0,T))} + \|\na_{x,t} (u_1-u_2)\|_{H^{0,2}(\Gamma\times(0,T))} + \|p_1-p_2\|_{H^{\frac{1}{2},2}(\Gamma\times(0,T))} \\
&\hspace{0.8cm} + \|H_1-H_2\|_{H^{0,2}(\Gamma\times(0,T))} + \|\na_{x,t} (H_1-H_2)\|_{H^{0,2}(\Gamma\times(0,T))} .
\end{align*}

In order to prove the stability results, we use the technique of Carleman estimate. In the next part, we will establish two Carleman inequalities which are the key points for the proof. 

\section{Carleman estimates}
\subsection{Carleman estimates with a singular weight function}
First of all, let's fix the weight function. Throughout this article, we use a singular weight function. Arbitrarily fix $t_0\in (0,T)$ and set $\delta := \min\{t_0, T-t_0 \}$. Let $l\in C^\infty [0,T]$ satisfy:
\begin{equation}
\label{con:choice-l}
\left\{
\begin{aligned}
&\ l(t) > 0, \qquad \qquad \quad 0 < t < T, \\
&\ l(t) =
\left\{
\begin{aligned}
&t, \qquad \qquad 0\le t\le \frac{\delta}{2}, \\
&T - t, \qquad T - \frac{\delta}{2}\le t\le T,
\end{aligned}
\right. \\
&\ l(t_0) > l(t), \quad \quad \qquad \forall t\in (0,T)\setminus \{ t_0\}.
\end{aligned}
\right.
\end{equation}
Then we can choose $e^{2s\alpha}$ as our weight function where
\begin{equation}
\label{con:choice-alpha-phi}
\vp(x,t) = \f{e^{\la d(x)}}{l(t)}, \quad
\al(x,t) = \f{e^{\la d(x)}-e^{2\la\|d\|_{C(\ov{\Om})}}}{l(t)}.
\end{equation}
This is called a singular weight because $\alpha$ tends to $-\infty$ as $t$ goes to $0$ and $T$. Thus, the weight is close to $0$ near $t=0,T$. 

Now we establish two key Carleman inequalities. The first one is for direct problem. We consider the following linearized MHD system:
\begin{equation}
\label{sy:MHD4}
\left\{
\begin{aligned}
&\pa_t u - \nu\De u + (B^{(1)}\cdot\na) u + (u\cdot\na) B^{(2)} + \na(B^{(3)}\cdot u) + L_1(H) + \na p = F &\quad in \ Q, \\
&\pa_t H - \kappa\De H + (D^{(1)}\cdot\na) H + (H\cdot\na) D^{(2)} + D^{(3)}\times \mathrm{rot}\; H + L_2(u) = G &\quad in \ Q, \\
& \mathrm{div}\; u = 0, \quad \mathrm{div}\; H = 0 &\quad in \ Q.
\end{aligned}
\right.
\end{equation}
Here
$$
\begin{aligned}
&L_1(H) = (C^{(1)}\cdot\na) H + (H\cdot\na) C^{(2)} + \na (C^{(3)}\cdot H), \\
&L_2(u) = (C^{(4)}\cdot\na) u + (u\cdot\na) C^{(5)},
\end{aligned}
$$
$\nu,\kappa\in W^{1,\infty}(Q)$ admit a positive lower bound and the coefficients $B^{(k)},C^{(k)},D^{(k)}$, $k\in \Bbb{N}$ are assumed to have enough regularity (e.g. $W^{2,\infty}(Q)$). For simplicity, we define
$$
\begin{aligned}
\|(u,p,H)\|_{\chi_s(Q)}^2 := \int_Q \bigg\{ &\frac{1}{s^2\vp^2}\bigg( |\pa_t u|^2 + \sum_{i,j=1}^3 |\pa_i \pa_j u|^2\bigg) + |\na u|^2 + s^2\vp^2|u|^2 + \frac{1}{s\vp}|\na p|^2 + s\vp|p|^2 \\
& + \frac{1}{s^2\vp^2}\bigg( |\pa_t H|^2 + \sum_{i,j=1}^3 |\pa_i \pa_j H|^2\bigg) + |\na H|^2 + s^2\vp^2|H|^2 \bigg\}e^{2s\alpha}dxdt.
\end{aligned}
$$
In the proof, we have further assumption that
\begin{equation}
\label{con:wk-div-zero}
\mathrm{div}\;\pa_t u = 0, \quad \mathrm{div}\;\Delta u = 0 \quad in \ Q.
\end{equation}
\noindent Condition (\ref{con:wk-div-zero}) should be true at least in the weak sense. In fact, if we have higher regularity of source terms $F$ and $G$, then we have improved regularity of the solution $u$. In that case, (\ref{con:wk-div-zero}) holds automatically after the condition $\mathrm{div}\;u = 0, \ in \ Q$.
\vspace{0.2cm}

Then the first Carleman estimate can be stated as:
\begin{thm}
\label{thm:CEDP}
Let $d\in C^2(\ov{\Om})$ satisfy (\ref{con:d}) and $F,G\in L^2(Q)$. Then for large fixed $\lambda$, there exist constants $s_0>0$ and $C>0$ such that
\begin{equation}
\begin{aligned}
\|&(u,p,H)\|_{\chi_s(Q)}^2 \le \;C\int_Q \big(|F|^2 + |G|^2\big) e^{2s\alpha}dxdt + Ce^{-s}\bigg(\|u\|_{L^2(\Sigma)}^2  + \|\na_{x,t} u\|_{L^2(\Sigma)}^2\\
&\hspace{4cm} + \|H\|_{L^2(\Sigma)}^2 + \|\na_{x,t} H\|_{L^2(\Sigma)}^2 + \|p\|_{L^2(0,T;H^{\frac{1}{2}}(\pa \Om))}^2 \bigg)
\end{aligned}
\end{equation}
for all $s\ge s_0$ and all $(u,p,H)\in H^{2,1}(Q)\times H^{1,0}(Q)\times H^{2,1}(Q)$ satisfying the system (\ref{sy:MHD4}).
\end{thm}
\noindent $\mathbf{Remarks.}$ (\romannumeral1) There is a confusion for $\|p\|_{L^2(\Om)}$ because $p$ can be changed up to a constant. Therefore, in this article, we actually mean $\inf_{c\in\Bbb{R}}\|p+c\|_{L^2(\Om)}$ while we just write $\|p\|_{L^2(\Om)}$.

(\romannumeral2) In this article, $C$ usually denotes generic positive constant which depends on $T,\Om$ and the coefficients but is independent of large parameter $s$ and $\la$ as well. However, $\lambda$ plays an important role in the proof of Carleman estimate. And so while the generic constant $C$ depends on $\la$, we use notation $C(\la)$ to indicate the dependence.
\vspace{0.2cm}

We prove Theorem \ref{thm:CEDP} by some techniques and combinations of Carleman estimates. Our key point is the estimate of pressure $p$. Thanks to the paper of $H^{-1}$- Carleman estimate for elliptic type (see Imanuvilov and Puel \cite{IP03}), we are able to establish the Carleman estimate with boundary data by a simple extension.

\begin{proof}[Proof of Theorem \ref{thm:CEDP}]
We divide the proof into three steps.
\vspace{0.2cm}

\noindent $\mathbf{First\ step.}$ We prove a Carleman estimate for pressure $p$ with boundary data.

We shall use the following lemma.
\begin{lem}
\label{lem:Ce-H1-zero-bdy}
Let $d\in C^2(\ov{\Om})$ be chosen as (\ref{con:d}) and $y\in H^1(\Om)$ satisfy
$$
\left\{
\begin{aligned}
&\ \De y + \sum_{j=1}^3 b_j(x) \pa_j y = f_0 + \sum_{j=1}^3\pa_j f_j &\quad in \ \Om, \\
&\ y = 0 &\quad on \ \pa\Om
\end{aligned}
\right.
$$
with $f_0,f_j\in L^2(\Om)$ and $b_j\in L^{\infty}(\Om)$, $j=1,2,3$. Then there exist constants $\la_0 \ge 1$, $s_0 \ge 1$ and $C>0$ such that
\begin{equation}
\label{eq:Ce-ellip-zero-bdy}
\begin{aligned}
\int_\Om \big( |\na y|^2 &+ s^2\la^2e^{2\la d}|y|^2\big) e^{2se^{\la d}}dx \\
\le \; &C\bigg (\int_\Om \f{1}{s\la^2}e^{-\la d}|f_0|^2e^{2se^{\la d}}dx + \sum_{j=1}^3 \int_\Om se^{\la d}|f_j|^2e^{2se^{\la d}}dx \bigg)
\end{aligned}
\end{equation}
for all $\la \ge \la_0$ and $s\ge s_0$.
\end{lem}
\begin{proof}[Proof of Lemma \ref{lem:Ce-H1-zero-bdy}]

We use the same technique as we choose the function $d$ and apply an $H^{\!-1}$ -Carleman estimate for elliptic type.

We take the zero extensions of $y,f_0,f_j,j=1,2,3$ to $\Om_1$ and denote them by the same letters. Here $\Om_1$ is chosen as that in (\ref{con:choice-Om1}). Thus we have
\begin{equation}
\label{sy:prf-ellip-zero-bdy}
\begin{aligned}
\Delta y + \sum_{j=1}^3 b_j(x) \pa_j y = f_0 + \sum_{j=1}^3\pa_j f_j \quad in \ \Om_1, \qquad y = 0 \quad on \ \pa\Om_1.
\end{aligned}
\end{equation}
Note that the function $d$ is chosen as (\ref{con:d}). We apply an $H^{\!-1}$- Carleman estimate (see Theorem A.1 of \cite{IP03}) to (\ref{sy:prf-ellip-zero-bdy}) to obtain
$$
\begin{aligned}
\int_{\Om_1} \big( |\na y|^2 &+ s^2\la^2e^{2\la d}|y|^2\big) e^{2se^{\la d}}dx \\
\le \; &C\bigg (\int_{\Om_1} \f{1}{s\la^2}e^{-\la d}|f_0|^2e^{2se^{\la d}}dx + \sum_{j=1}^3 \int_{\Om_1} se^{\la d}|f_j|^2e^{2se^{\la d}}dx \bigg)
\end{aligned}
$$
for all $\la\ge \la_0$ and $s\ge s_0$. In $H^{\!-1}$- Carleman estimate, there is a term of integral over interior sub-domain $\om$. However, we remove this term in the above inequality because we have chosen $\om\subset\subset \Om_1$ such that $\ov{\om}\subset\Om_1\setminus \ov{\Om}$ and $y$ vanishes outside of $\Om$. Since $f_0,f_j,j=1,2,3$ are also zero outside of $\Om$, (\ref{eq:Ce-ellip-zero-bdy}) is proved.
\end{proof}
\vspace{0.2cm}

We apply operator div to the first equation in (\ref{sy:MHD4}). By condition (\ref{con:wk-div-zero}),
$$
\Delta p = \mathrm{div} \big(F - L_1(H) - (B^{(1)}\cdot\na)u - (u\cdot\na)B^{(2)} - \na(B^{(3)}\cdot u)\big)
$$
holds at least in the weak sense.
By Sobolev Trace Theorem, there exists $\wt{p}\in H^1(\Om)$ such that
$$
\wt{p} = p \quad on \ \pa\Om
$$
and
\begin{equation}
\label{eq:Stt1}
\|\wt{p}\|_{H^1(\Om)} \le C\|\wt{p}\|_{H^{\frac{1}{2}}(\pa\Om)} = C\|p\|_{H^{\frac{1}{2}}(\pa\Om)}.
\end{equation}
We then set
$$
q = p - \wt{p} \quad in \ \Om.
$$
Thus we have
\begin{equation}
\label{sy:prf-ellip-ext-zero-bdy}
\left\{
\begin{aligned}
&\ \Delta q = \mathrm{div}(F - L_1(H) - (B^{(1)}\cdot\na)u - (u\cdot\na)B^{(2)} - \na(B^{(3)}\cdot u) - \na \wt{p}) &\quad in \ \Om,\ \ \\
&\ q = 0 &\quad on \ \pa\Om.
\end{aligned}
\right.
\end{equation}
Applying Lemma \ref{lem:Ce-H1-zero-bdy} to (\ref{sy:prf-ellip-ext-zero-bdy}), we obtain
$$
\begin{aligned}
\int_\Om &\big( |\na q|^2 + s^2\la^2e^{2\la d}|q|^2\big) e^{2se^{\la d}}dx \\
&\le \; C\int_\Om se^{\la d}|F|^2e^{2se^{\la d}}dx + C\int_\Om se^{\la d}(|\na u|^2 + |u|^2 + |\na H|^2 + |H|^2 + |\na \wt{p}|^2)e^{2se^{\la d}}dx
\end{aligned}
$$
for all $\la\ge \la_0$ and all $s\ge s_0$.
Since $p = q + \wt{p}$, we have
\begin{equation}
\label{eq:prf-ce-1}
\begin{aligned}
\int_\Om \big( |\na p|^2 &+ s^2\la^2e^{2\la d}|p|^2\big) e^{2se^{\la d}}dx \\
\le \; &2\int_\Om \big( |\na q|^2 + s^2\la^2e^{2\la d}|q|^2\big) e^{2se^{\la d}}dx + 2\int_\Om \big( |\na \wt{p}|^2 + s^2\la^2e^{2\la d}|\wt{p}|^2\big) e^{2se^{\la d}}dx \\
\le \; &C\int_\Om se^{\la d}|F|^2e^{2se^{\la d}}dx + Cs^2\la^2e^{2\la \|d\|_{C(\ov{\Om})}}e^{2se^{\la \|d\|_{C(\ov{\Om})}}}\|p\|_{H^{\frac{1}{2}}(\pa\Om)}^2 \\
&+ C\int_\Om se^{\la d}(|\na u|^2 + |u|^2 + |\na H|^2 + |H|^2)e^{2se^{\la d}}dx \\
\end{aligned}
\end{equation}
for all $\la\ge \la_0$ and all $s\ge s_0$. We used (\ref{eq:Stt1}) in the last inequality.

Recall the definition of weight function (\ref{con:choice-l})-(\ref{con:choice-alpha-phi}). Let $s\ge s_1\equiv s_0l(t_0)$. Then $sl^{-1}(t)\ge s_0$ for all $0\le t\le T$. Hence substituting $s$ by $sl^{-1}(t)$ in (\ref{eq:prf-ce-1}) yields
$$
\begin{aligned}
\int_\Om \big( |\na p|^2 &+ s^2\la^2\vp^2|p|^2\big) e^{2s\vp}dx \le \; C\int_\Om s\vp|F|^2e^{2s\vp}dx + Cs^2\la^2 l^{-2}e^{2\la}e^{2sl^{-1}e^{\la}}\|p\|_{H^{\frac{1}{2}}(\pa\Om)}^2 \\
&+ C\int_\Om s\vp(|\na u|^2 + |u|^2 + |\na H|^2 + |H|^2)e^{2s\vp}dx
\end{aligned}
$$
Without loss of generality, we can assume $\|d\|_{C(\ov{\Om})}=1$ here. Multiplying the above inequality by $s^{-1}l(t)e^{-2sl^{-1}(t)e^{2\la}}$ and integrating over $(0,T)$, we obtain
\begin{equation}
\label{eq:Ce-sin-p}
\begin{aligned}
&\int_Q \big( \frac{e^{\la d}}{s\vp}|\na p|^2 + s\la^2\vp e^{\la d}|p|^2\big) e^{2s\alpha}dxdt \le \; C\int_Q e^{\la d}|F|^2e^{2s\alpha}dxdt + C(\la)e^{-s}\|p\|_{L^2(0,T;H^{\frac{1}{2}}(\pa\Om))}^2 \\
&\hspace{5cm} + C\int_Q e^{\la d}(|\na u|^2 + |u|^2 + |\na H|^2 + |H|^2)e^{2s\alpha}dxdt
\end{aligned}
\end{equation}
for all $\la\ge \la_0$ and all $s\ge s_1$.
\vspace{0.2cm}

\noindent $\mathbf{Second\ step.}$ We apply a Carleman estimate for parabolic type.

We have the following lemma.
\begin{lem}
\label{lem:Ce-sin-para}
Let $\vp$ be chosen as (\ref{con:choice-alpha-phi}) and $y\in H^{2,1}(Q)$ satisfy
$$
\begin{aligned}
\quad \pa_t y - \nu(x,t)\De y + \sum_{j=1}^3 b_j(x,t)\pa_j y + c(x,t)y = f \quad in \; Q
\end{aligned}
$$
with $\nu,b_j,c\in W^{1,\infty}(Q)$, $\nu\ge c_0>0$ and $f\in L^2(Q)$, $j=1,2,3$. Then there exist  constants $\la_0>0$, $s_0>0$ and $C>0$ such that
\begin{equation}
\begin{aligned}
\int_Q \bigg\{ \frac{e^{\la d}}{s^2\vp^2}\bigg( |\pa_t y|^2 + \sum_{i,j=1}^3 &|\pa_i \pa_j y|^2 \bigg) + \la^2e^{\la d}|\na y|^2 + s^2\la^4\vp^2e^{\la d} |y|^2\bigg\} e^{2s\al}dxdt \\
\le \; &C\int_Q \f{e^{\la d}}{s\vp}|f|^2e^{2s\al}dxdt + C(\la)e^{-s}\int_{\Sigma} (|y|^2 + |\na_{x,t} y|^2)dSdt
\end{aligned}
\end{equation}
for all $\la\ge \la_0$ and all $s\ge s_0$.
\end{lem}

This proof is similar to that in Chae, Imanuvilov and Kim \cite{CIK96}. See also Imanuvilov \cite{I95}.
\vspace{0.2cm}

We rewrite the first equation in (\ref{sy:MHD4}) to get
\begin{equation*}
\pa_t u - \nu \Delta u + (B^{(1)}\cdot\na)u + (u\cdot\na)B^{(2)} + \na(B^{(3)}\cdot u)= F - \na p - L_1(H).
\end{equation*}
Applying Lemma \ref{lem:Ce-sin-para} to each component of above equations, we obtain
\begin{equation}
\label{eq:Ce-sin-u}
\begin{aligned}
&\int_Q \bigg\{ \frac{e^{\la d}}{s^2\vp^2}\bigg( |\pa_t u|^2 + \sum_{i,j=1}^3 |\pa_i \pa_j u|^2 \bigg) + \la^2e^{\la d}|\na u|^2 + s^2\la^4\vp^2e^{\la d} |u|^2\bigg\} e^{2s\al}dxdt \le C\int_Q \f{e^{\la d}}{s\vp}|F|^2e^{2s\al}dxdt \\
&\hspace{1cm} + C\int_Q \f{e^{\la d}}{s\vp}(|\na p|^2 + |\na H|^2 + |H|^2)e^{2s\al}dxdt + C(\la)e^{-s}\big(\|u\|_{L^2(\Sigma)}^2 + \|\na_{x,t} u\|_{L^2(\Sigma)}^2\big)
\end{aligned}
\end{equation}
for all $\la\ge \la_1$ and all $s\ge s_2$.

Next, we apply Carleman estimate for parabolic type to the second equation of (\ref{sy:MHD4}) and we have the following estimate:
\begin{equation}
\label{eq:Ce-sin-H}
\begin{aligned}
&\int_Q \bigg\{ \frac{e^{\la d}}{s^2\vp^2}\bigg( |\pa_t H|^2 + \sum_{i,j=1}^3 |\pa_i \pa_j H|^2 \bigg) + \la^2e^{\la d} |\na H|^2 + s^2\la^4\vp^2e^{\la d} |H|^2\bigg\} e^{2s\al}dxdt \le C\int_Q |G|^2e^{2s\al}dxdt \\
&\hspace{1cm} + C\int_Q \f{e^{\la d}}{s\vp}(|\na u|^2 + |u|^2)e^{2s\al}dxdt + C(\la)e^{-s}\big(\|H\|_{L^2(\Sigma)}^2 + \|\na_{x,t} H\|_{L^2(\Sigma)}^2\big)
\end{aligned}
\end{equation}
for all $\la\ge \la_2$ and all $s\ge s_3$. Here we used $s^{-1}\vp^{-1}e^{\la d}\le 1\ in\; Q$ for any $s\ge s_1$.
\vspace{0.2cm}

\noindent $\mathbf{Third\ step.}$ We combine the estimates for $p,u$ and $H$.

Combining (\ref{eq:Ce-sin-p}), (\ref{eq:Ce-sin-u}) and (\ref{eq:Ce-sin-H}), we obtain
$$
\begin{aligned}
& \int_Q \bigg\{ \frac{e^{\la d}}{s^2\vp^2}\bigg( |\pa_t u|^2 + \sum_{i,j=1}^3 |\pa_i \pa_j u|^2\bigg) + \la^2e^{\la d}|\na u|^2 + s^2\la^4\vp^2e^{\la d} |u|^2 + \frac{e^{\la d}}{s\vp}|\na p|^2 + s\la^2\vp e^{\la d} |p|^2 \\
&\hspace{1cm} + \frac{e^{\la d}}{s^2\vp^2}\bigg( |\pa_t H|^2 + \sum_{i,j=1}^3 |\pa_i \pa_j H|^2\bigg) + \la^2e^{\la d} |\na H|^2 + s^2\la^4\vp^2e^{\la d} |H|^2 \bigg\}e^{2s\alpha}dxdt \\
& \le \;C\int_Q e^{\la d}( |F|^2 + |G|^2 ) e^{2s\alpha}dxdt + C\int_Q e^{\la d}\big( |\na u|^2 + |u|^2 + |\na H|^2 + |H|^2 \big) e^{2s\alpha}dxdt \\
&\hspace{1cm} + C(\la)e^{-s}\bigg(\|u\|_{L^2(\Sigma)}^2 + \|\na_{x,t} u\|_{L^2(\Sigma)}^2 + \|H\|_{L^2(\Sigma)}^2 + \|\na_{x,t} H\|_{L^2(\Sigma)}^2 + \|p\|_{L^2(0,T;H^{\frac{1}{2}}(\pa\Om))}^2 \bigg)
\end{aligned}
$$
for all $\la\ge \la_2$ and all $s\ge s_3$. Finally we can fix $\la$ large enough to absorb the second term on the right-hand side into the left-hand side. By the relations $e^{\la d}\ge 1,\la \ge 1$, we obtain
$$
\begin{aligned}
&\|(u,p,H)\|_{\chi_s(Q)}^2 \le \;C(\la)\int_Q \big(|F|^2 + |G|^2\big) e^{2s\alpha}dxdt \\
&\hspace{1cm} + C(\la)e^{-s}\bigg(\|u\|_{L^2(\Sigma)}^2 + \|\na_{x,t} u\|_{L^2(\Sigma)}^2 + \|H\|_{L^2(\Sigma)}^2 + \|\na_{x,t} H\|_{L^2(\Sigma)}^2 + \|p\|_{L^2(0,T;H^{\frac{1}{2}}(\pa\Om))}^2 \bigg)
\end{aligned}
$$
for fixed $\la$ large enough and all $s \ge s_4\equiv max\{s_1,s_2,s_3\}$.
\vspace{0.2cm}

The proof of Theorem \ref{thm:CEDP} is completed.
\end{proof}

On the other hand, we investigate the following two first-order partial differential operators: 
\vspace{0.2cm}

($\romannumeral1)\quad P f := \mathrm{div} (fA)=A \na f + f\mathrm{div} A, \quad f\in H^1(\Om)$, 
\vspace{0.2cm}

($\romannumeral2)\quad Q g := \mathrm{rot} (gb)=\na g\times b + g\mathrm{rot} b, \quad g\in H^1(\Om)$
\vspace{0.2cm}

\noindent where  $A=(A_{ij})_{i,j}$ is a $3\times 3$ matrix and $b=(b_1,b_2,b_3)^T$ is a vector satisfying $A\in W^{1,\infty}(\Om), b\in W^{2,\infty}(\Om)$. Recall that the divergence of a matrix is defined as $[\mathrm{div} A]_k=\sum_{j=1}^3 \pa_j A_{kj}$. We have the following Carleman inequalities: 
\begin{thm}
\label{thm:CEIP}
Let $d$ be chosen as (\ref{con:d}) and $\vp_0:=e^{\la d}$. Assume that 
$$
\mathrm{det} A(x) \neq 0 \ and \ |\na d(x) \times b(x)|\neq 0, \qquad for\ x\in \ov{\Om}.
$$
Then there exist constants $\la_0\ge 1$, $s_0\ge 1$ and a generic constant $C>0$ such that
\begin{equation}
\label{CEIP1}
\int_{\Om} (|\na f|^2 + s^2\la^2\vp_0^2 |f|^2)e^{2s\vp_0} dx \le C\int_{\Om} |P f|^2 e^{2s\vp_0}dx + C\int_{\Gamma} s\la\vp_0 |f|^2e^{2s\vp_0}d\sigma
\end{equation}
and
\begin{equation}
\label{CEIP2}
\begin{aligned}
&\int_{\Om} (|\na g|^2 + s^2\la^2\vp_0^2 |g|^2)e^{2s\vp_0} dx \le C\int_{\Om} (\frac{1}{s^2\la^2\vp_0^2}|\na (Qg)|^2 + |Qg|^2) e^{2s\vp_0}dx \\
&\hspace{5cm} + C\int_{\partial\Om} (\frac{1}{s\la\vp_0}|\na g|^2 + s\la\vp_0 |g|^2)e^{2s\vp_0}d\sigma
\end{aligned}
\end{equation}
for all $\la\ge \la_0$, $s\ge s_0$ and $f\in H^1(\Om), g\in H^2(\Om)$. 
\end{thm}
To proof these inequalities, we apply the idea of Lemma 6.1 in \cite{Y09}. 
\begin{proof}
We first prove inequality (\ref{CEIP1}). Set $w=fe^{s\vp_0}$. Then 
$$
Pf=P(we^{-s\vp_0})=e^{-s\vp_0}(A\na w + w\mathrm{div} A - s\la\vp_0 (A\na d) w).
$$
We rewrite it in components, that is
\begin{equation}
\label{eq1}
[Pf]_ke^{s\vp_0}=\sum_{j=1}^3 \big( A_{kj}\pa_j w + \pa_j A_{kj} w - s\la\vp_0 (A_{kj}\pa_j d)w\big)
\end{equation}
Now choose $a=(a_1,a_2,a_3)^T\in L^\infty (\Om)$ such that $\sum_{k=1}^3 a_k A_{kj} = \pa_j d$ for any $x\in \ov{\Om}$. In fact, the existence of such $\{a_k\}_{k=1,2,3}$ comes from the assumption $\mathrm{det} A\neq 0$ on $\ov{\Om}$. 

We multiply $a_k$ to equation (\ref{eq1}) and take summation over $k$:
$$
\sum_{k=1}^3 a_k[Pf]_k e^{s\vp_0} = \na d\cdot \na w + \big(\sum_{j,k=1}^3 a_k \pa_j A_{kj}\big) w - s\la\vp_0 |\na d|^2 w
$$
Then we estimate
$$
\begin{aligned}
&\int_\Om \Big|\sum_{k=1}^3 a_k[Pf]_k \Big|^2e^{2s\vp_0}dx = \int_\Om s^2\la^2\vp_0^2|\na d|^4|w|^2 dx + \int_\Om |\na d \cdot \na w + (a\cdot \mathrm{div} A) w|^2 dx \\
&\hspace{4.3cm} - 2\int_\Om s\la\vp_0 |\na d|^2(\na d\cdot \na w + (a\cdot\mathrm{div}A) w)w dx \\
&\hspace{3.9cm}\ge \int_\Om s^2\la^2\vp_0^2|\na d|^4|w|^2 dx - 2\int_\Om s\la\vp_0 |\na d|^2 (a\cdot \mathrm{div} A)|w|^2 dx \\
&\hspace{4.3cm} - \int_{\pa\Om} s\la\vp_0 |\na d|^2 \frac{\pa d}{\pa n}|w|^2 d\sigma + \int_\Om s\la \mathrm{div} (\vp_0|\na d|^2 \na d)|w|^2 dx \\
&\hspace{3.9cm}\ge \int_\Om s^2\la^2\vp_0^2|\na d|^4|w|^2 dx - \int_{\Gamma} s\la\vp_0 |\na d|^2 \frac{\pa d}{\pa n}|w|^2 d\sigma \\
&\hspace{4.3cm} + \int_\Om s\la\vp_0 \big(\la|\na d|^4 + \na |\na d|^2 \na d + |\na d|^2 (\Delta d - 2(a\cdot\mathrm{div} A))\big)|w|^2 dx. 
\end{aligned}
$$
In the last inequality, we used the relation (\ref{con:d}) to get $\frac{\pa d}{\pa n}<0$ on $\pa\Om\setminus \Gamma$. By choose $\la$ large, we can absorb the third term on the right-hand side. Thus,
\begin{equation}
\label{eq2}
\int_\Om s^2\la^2\vp_0^2|f|^2e^{2s\vp_0} dx \le C \int_\Om |Pf|^2e^{2s\vp_0}dx + C \int_\Gamma s\la\vp_0|f|^2e^{2s\vp_0} d\sigma
\end{equation} 
holds for all $\la\ge \la_1$ and $s\ge 1$. 

Furthermore, for $l=1,2,3$, we can also choose $a^{(l)}=(a_1^{(l)},a_2^{(l)},a_3^{(l)})^T\in L^\infty (\Om)$ such that $\sum_{k=1}^3 a_k^{(l)} A_{kj}=\delta_{lj}$. Take summation over $k$ after multiply $a_k^{(l)}$ to (\ref{eq1}):
$$
\sum_{k=1}^3 a_k^{(l)}[Pf]_k e^{s\vp_0} = \pa_l w + \big(\sum_{j,k=1}^3 a_k^{(l)} \pa_j A_{kj}\big) w - s\la\vp_0 (\pa_l d) w
$$
Again we estimate
$$
\begin{aligned}
&\int_\Om \Big|\sum_{k=1}^3 a_k^{(l)}[Pf]_k \Big|^2e^{2s\vp_0}dx = \int_\Om |\pa_l w|^2 dx + \int_\Om |(a^{(l)}\cdot \mathrm{div} A) - s\la\vp_0(\pa_l d)|^2 |w|^2 dx \\
&\hspace{4.3cm} + 2\int_\Om \big((a^{(l)}\cdot \mathrm{div} A) - s\la\vp_0(\pa_l d)\big)w(\pa_l w) dx \\
&\hspace{3.9cm}\ge \int_\Om |\pa_l w|^2 dx + 2\int_\Om (a^{(l)}\cdot \mathrm{div} A)w(\pa_l w) dx \\
&\hspace{4.3cm} - \int_{\pa\Om} s\la\vp_0 (\pa_l d)n_l |w|^2 d\sigma + \int_\Om s\la\vp_0(\la|\pa_l d|^2 + \pa_l^2 d)|w|^2 dx.
\end{aligned}
$$
Rewrite the above inequality and take summation over $l$ on both sides:
$$
\begin{aligned}
&\int_\Om |\na w|^2 dx \le \int_\Om \sum_{l=1}^3\Big|\sum_{k=1}^3 a_k^{(l)}[Pf]_k \Big|^2e^{2s\vp_0}dx + \int_{\pa\Om} s\la\vp_0 \frac{\pa d}{\pa n} |w|^2 d\sigma \\
&\hspace{2.5cm} - 2\sum_{l=1}^3\int_\Om (a^{(l)}\cdot \mathrm{div} A)w(\pa_l w) dx - \int_\Om s\la\vp_0(\la|\na d|^2 + \Delta d)|w|^2 dx \\
&\hspace{1.8cm} \le C\int_\Om |Pf|^2e^{2s\vp_0}dx + \int_\Gamma s\la\vp_0\frac{\pa d}{\pa n} |w|^2 d\sigma \\
&\hspace{2.5cm} + \frac{1}{2}\int_\Om |\na w|^2 dx + 2\int_\Om \sum_{l=1}^3 |a^{(l)}\cdot\mathrm{div} A|^2 |w|^2 dx + \int_\Om s\la\vp_0|w|^2 dx
\end{aligned}
$$
This leads to
$$
\int_\Om |\na w|^2 dx \le C\int_\Om |Pf|^2e^{2s\vp_0}dx + C\int_\Gamma s\la\vp_0 |w|^2 d\sigma + C\int_\Om s\la\vp_0|w|^2 dx
$$
Together with (\ref{eq2}) and take $\la$ large enough to absorb the last term on the right-hand side. Finally, we obtain
$$
\int_{\Om} (|\na f|^2 + s^2\la^2\vp_0^2 |f|^2)e^{2s\vp_0} dx \le C\int_{\Om} |P f|^2 e^{2s\vp_0}dx + C\int_{\Gamma} s\la\vp_0 |f|^2e^{2s\vp_0}d\sigma
$$
for all $\la\ge \la_2$ and $s\ge 1$. 

Next we consider the operator $Q$. 
Set $v=ge^{s\vp_0}$. Then
$$
Qg=Q(ve^{-s\vp_0})=e^{-s\vp_0}(\na v\times b + (\mathrm{rot} b)v - s\la\vp_0(\na d\times b)v).
$$
There is no hope to do in the same way as for operator $P$. In fact, we denote
$$
B=
\begin{pmatrix}
0 & b_3 & -b_2 \\
-b_3 & 0 & b_1 \\
b_2 & -b_1 & 0 \\
\end{pmatrix}.
$$
Then we rewrite the above formula:
$$
Qg e^{s\vp_0} = B\na v + (\mathrm{rot} b)v - s\la\vp_0(B\na d)v.
$$
However, $\mathrm{det} B = b_1b_2b_3 + (-b_1b_2b_3)=0$. Thus, we calculate directly
$$
\begin{aligned}
&\int_\Om |Qg|^2e^{2s\vp_0} dx = \int_\Om s^2\la^2\vp_0^2 |B\na d|^2 |v|^2 dx + \int_\Om |B\na v + (\mathrm{rot} b)v|^2 dx \\
&\hspace{2.7cm} -2\int_\Om s\la\vp_0 (B\na d)\cdot (B\na v + (\mathrm{rot} b)v)v dx \\
&\hspace{2.5cm} \ge \int_\Om s^2\la^2\vp_0^2 |B\na d|^2 |v|^2 dx -2\int_\Om s\la\vp_0(B\na d)\cdot (\mathrm{rot} b)|v|^2 dx \\
&\hspace{2.7cm} -\int_{\pa\Om} s\la\vp_0 (B\na d)\cdot (Bn)|v|^2 d\sigma + \int_\Om s\la\vp_0 \big(\la |B\na d|^2 + \mathrm{div}(B^T(B\na d))\big) |v|^2 dx
\end{aligned}
$$
By noting the assumption that $|B\na d|=|\na d\times b| \neq 0$ in $\ov{\Om}$, we can take $\la$ large to absorb the second and fourth terms on the right-hand side:
\begin{equation}
\label{eq3}
\int_\Om s^2\la^2\vp_0^2 |g|^2e^{2s\vp_0} dx \le C\int_\Om |Qg|^2e^{2s\vp_0} dx + C\int_{\pa\Om}s\la\vp_0 |g|^2e^{2s\vp_0} d\sigma
\end{equation}
for all $\la\ge \la_3$ and $s\ge 1$. 

We take the $k$-th derivative of (\romannumeral2) and denote $g_k=\pa_k g$. Define
$$
Q_k g_k:= \pa_k(Qg) - \na g \times \pa_k b - g(\mathrm{rot}(\pa_k b))= \na g_k\times b + g_k(\mathrm{rot}b).
$$
By applying similar argument above to operator $Q_k$, we have 
$$
\begin{aligned}
&\int_\Om |g_k|^2e^{2s\vp_0} dx \le C\int_\Om \frac{1}{s^2\la^2\vp_0^2}|Q_k g_k|^2e^{2s\vp_0} dx + C\int_{\pa\Om} \frac{1}{s\la\vp_0} |g_k|^2e^{2s\vp_0} d\sigma \\
&\hspace{0cm} \le C\int_\Om \frac{1}{s^2\la^2\vp_0^2}|\pa_k (Qg)|^2e^{2s\vp_0} dx + C\int_\Om \frac{1}{s^2\la^2\vp_0^2}(|\na g|^2 +|g|^2)e^{2s\vp_0}dx + C\int_{\pa\Om}\frac{1}{s\la\vp_0} |g_k|^2e^{2s\vp_0} d\sigma
\end{aligned}
$$
for all $\la \ge \la_4$, $s\ge 1$ and $k=1,2,3$. Sum up the estimates over $k$ and absorb again the lower-order terms by taking $\la$ large:
\begin{equation}
\label{eq4}
\begin{aligned}
&\int_\Om |\na g|^2e^{2s\vp_0} dx \le C\int_\Om \frac{1}{s^2\la^2\vp_0^2}|\na (Qg)|^2e^{2s\vp_0} dx + C\int_\Om \frac{1}{s^2\la^2\vp_0^2}|g|^2e^{2s\vp_0}dx \\
&\hspace{2.8cm}+ C\int_{\pa\Om} \frac{1}{s\la\vp_0} |\na g|^2e^{2s\vp_0} d\sigma
\end{aligned}
\end{equation}
for all $\la \ge \la_5$ and all $s\ge 1$. Combining (\ref{eq3}) and (\ref{eq4}), we proved (\ref{CEIP2}) and also Theorem \ref{thm:CEIP} with $\la_0=max\{\la_i:\ 1\le i\le 5\}$ and $s_0=1$.   

\end{proof}
In \eqref{CEIP1} and \eqref{CEIP2}, we let $s_1 = s_0 l(t_0) = l(t_0)$. Then for all $s\ge s_1$, $sl^{-1}(t_0)\ge s_1l^{-1}(t_0) = s_0$. Substituting $s$ by $sl^{-1}(t_0)$ yields
\begin{align*}
\int_{\Om} (|\na f|^2 + s^2\la^2\vp^2(x,t_0) |f|^2)e^{2s\vp(x,t_0)} dx \le C\int_{\Om} |P f|^2 e^{2s\vp(x,t_0)}dx + C\int_{\Gamma} s\la\vp(x,t_0) |f|^2e^{2s\vp(x,t_0)}d\sigma
\end{align*}
and
\begin{align*}
&\int_{\Om} (|\na g|^2 + s^2\la^2\vp^2(x,t_0) |g|^2)e^{2s\vp(x,t_0)} dx \le C\int_{\Om} (\frac{1}{s^2\la^2\vp^2(x,t_0)}|\na (Qg)|^2 + |Qg|^2) e^{2s\vp(x,t_0)}dx \\
&\hspace{5cm} + C\int_{\partial\Om} (\frac{1}{s\la\vp(x,t_0)}|\na g|^2 + s\la\vp(x,t_0) |g|^2)e^{2s\vp(x,t_0)}d\sigma
\end{align*}
for all $\la\ge \la_0$ and all $s\ge s_1$. By multiplying $exp\{-2s\f{e^{2\la\|d\|}}{l(t_0)}\}$ on both inequalities, we derive 
\begin{thm}
\label{thm:CEIP-sin}
Under the assumptions that 
$$
\mathrm{det} A(x) \neq 0 \ and \ |\na d(x) \times b(x)|\neq 0, \qquad for\ x\in \ov{\Om},
$$
there exist constants $\la_0\ge 1$, $s_0\ge 1$ and a generic constant $C>0$ such that
\begin{align*}
\int_{\Om} (|\na f|^2 + s^2\la^2\vp^2(x,t_0) |f|^2)e^{2s\al (x,t_0)} dx \le C\int_{\Om} |P f|^2 e^{2s\al (x,t_0)}dx + C\int_{\Gamma} s\la\vp(x,t_0) |f|^2e^{2s\al (x,t_0)}d\sigma
\end{align*}
and
\begin{align*}
&\int_{\Om} (|\na g|^2 + s^2\la^2\vp^2(x,t_0) |g|^2)e^{2s\al (x,t_0)} dx \le C\int_{\Om} (\frac{1}{s^2\la^2\vp^2(x,t_0)}|\na (Qg)|^2 + |Qg|^2) e^{2s\al (x,t_0)}dx \\
&\hspace{5cm} + C\int_{\pa\Om} (\frac{1}{s\la\vp(x,t_0)}|\na g|^2 + s\la\vp(x,t_0) |g|^2)e^{2s\al (x,t_0)}d\sigma
\end{align*}
for all $\la\ge \la_0$, $s\ge s_0$ and $f\in H^1(\Om), g\in H^2(\Om)$. 
\end{thm}
\subsection{Carleman estimates with a regular weight function}
Throughout this part, we use a regular weight function. Arbitrarily fix $t_0\in (0,T)$ and set $\delta := \min\{t_0, T-t_0 \}$. Then we select our weight function as
\begin{equation}
\label{con:choice-phi-psi}
\vp(x,t) = e^{\la \psi(x,t)}, \quad
\psi(x,t) = d(x) - \beta(t-t_0)^2 + c_0
\end{equation}
where $d$ is the same choice as \eqref{con:d}, parameter $\beta>0$ to be fixed later and $c_0:=\max\{\beta t_0^2, \beta (T-t_0)^2\}$ so that $\psi$ is always nonnegative in $Q$.

Similar to the last subsection, we intend to establish two key Carleman inequalities with this regular weight. One is for direct problem and the other is for inverse problem. 
Firstly, we consider the following linearized MHD system:
\begin{equation}
\label{sy:MHD4n}
\left\{
\begin{aligned}
&\pa_t u - \nu\De u + (B^{(1)}\cdot\na) u + (u\cdot\na) B^{(2)} + \na(B^{(3)}\cdot u) + L_1(H) + \na p = F &\quad in \ Q, \\
&\pa_t H - \kappa\De H + (D^{(1)}\cdot\na) H + (H\cdot\na) D^{(2)} + D^{(3)}\times \mathrm{rot}\; H + L_2(u) = G &\quad in \ Q, \\
& \mathrm{div}\; u = h, &\quad in \ Q
\end{aligned}
\right.
\end{equation}
which is exactly system \eqref{sy:MHD4}. For simplicity, we define
\begin{align*}
\|(u,p,H)\|_{\sigma_s(Q)}^2 := \int_Q \bigg\{ &\frac{1}{s\vp}\bigg( |\pa_t u|^2 + \sum_{i,j=1}^3 |\pa_i \pa_j u|^2\bigg) + s\vp|\na u|^2 + s^3\vp^3|u|^2 + |\na p|^2 + s^2\vp^2|p|^2 \\
& + \frac{1}{s\vp}\bigg( |\pa_t H|^2 + \sum_{i,j=1}^3 |\pa_i \pa_j H|^2\bigg) + s\vp|\na H|^2 + s^3\vp^3|H|^2 \bigg\}e^{2s\vp} dxdt.
\end{align*}
Then we have the first Carleman estimate:
\begin{thm}
\label{thm:CEDPn}
Let $d\in C^2(\ov{\Om})$ satisfy \eqref{con:d} and $F,G\in L^2(Q)$. Then for large fixed $\lambda$, there exist constants $s_0>0$ and $C>0$ such that
\begin{align*}
\|&(u,p,H)\|_{\sigma_s(Q)}^2 \le \;C\int_Q s\vp\big(|F|^2 + |G|^2\big) e^{2s\vp}dxdt + C\int_Q s\vp|\na_{x,t} h|^2e^{2s\vp}dxdt \\
&\hspace{1cm} + Ce^{Cs}\Big(\|u\|_{L^2(\Sigma)}^2  + \|\na_{x,t} u\|_{L^2(\Sigma)}^2 + \|H\|_{L^2(\Sigma)}^2 + \|\na_{x,t} H\|_{L^2(\Sigma)}^2 + \|p\|_{L^2(0,T;H^{\frac{1}{2}}(\pa \Om))}^2 \Big)
\end{align*}
for all $s\ge s_0$ and all $(u,p,H)$ smooth enough and satisfying the system (\ref{sy:MHD4n}) with the conditions
\begin{align}
\label{con:0,T}
u(\cdot,0) = u(\cdot,T) = H(\cdot,0) = H(\cdot,T) = 0.
\end{align}
\end{thm}
\noindent $\mathbf{Remarks.}$ (\romannumeral1) There is a confusion for $\|p\|_{L^2(\Om)}$ because $p$ can be changed up to a constant. Therefore, in this article, we actually mean $\inf_{c\in\Bbb{R}}\|p+c\|_{L^2(\Om)}$ while we just write $\|p\|_{L^2(\Om)}$.

(\romannumeral2) In this article, $C$ usually denotes generic positive constant which depends on $T,\Om$ and the coefficients but is independent of large parameter $s$ and $\la$ as well. However, $\lambda$ plays an important role in the proof of Carleman estimate. And so while the generic constant $C$ depends on $\la$, we use notation $C(\la)$ to indicate the dependence.
\vspace{0.2cm}

We prove Theorem \ref{thm:CEDP} by some techniques and combinations of Carleman estimates. Our key point is the estimate of pressure $p$. Thanks to the paper of $H^{-1}$- Carleman estimate for elliptic type (see Imanuvilov and Puel \cite{IP03}), we are able to establish the Carleman estimate with boundary data by a simple extension.

\begin{proof}[Proof of Theorem \ref{thm:CEDP}]
We divide the proof into three steps.
\vspace{0.2cm}

\noindent $\mathbf{First\ step.}$ We prove a Carleman estimate for pressure $p$ with boundary data.

We apply operator div to the first equation in (\ref{sy:MHD4n}). Formal calculation leads to
\begin{align*}
&\Delta p = \mathrm{div} \big(F + \nu \na h - L_1(H) - (B^{(1)}\cdot\na)u - (u\cdot\na)B^{(2)} - \na(B^{(3)}\cdot u)\big) - \pa_t h \\
&\hspace{2cm} + \sum_{i,j=1}^3\pa_j ((\pa_i\nu) \pa_j u^i) - \sum_{i,j=1}^3 (\pa_i\pa_j\nu) \pa_j u^i
\end{align*}
By Sobolev Trace Theorem, there exists $\wt{p}\in H^1(\Om)$ such that
$$
\wt{p} = p \quad on \ \pa\Om
$$
and
\begin{equation}
\label{eq:Stt1n}
\|\wt{p}\|_{H^1(\Om)} \le C\|\wt{p}\|_{H^{\frac{1}{2}}(\pa\Om)} = C\|p\|_{H^{\frac{1}{2}}(\pa\Om)}.
\end{equation}
We then set
$$
q = p - \wt{p} \quad in \ \Om.
$$
Thus we have
\begin{equation}
\label{sy:prf-ellip-ext-zero-bdyn}
\left\{
\begin{aligned}
&\ \Delta q = \Delta p - \mathrm{div}(\na \wt{p}) &\qquad in \ \Om,\ \ \\
&\ q = 0 &\qquad on \ \pa\Om.
\end{aligned}
\right.
\end{equation}
Applying Lemma \ref{lem:Ce-H1-zero-bdy} to (\ref{sy:prf-ellip-ext-zero-bdyn}), we obtain
\begin{align*}
\int_\Om &\big( |\na q|^2 + s^2\la^2e^{2\la d}|q|^2\big) e^{2se^{\la d}}dx \le \; C\int_\Om se^{\la d}|F|^2e^{2se^{\la d}}dx + C\int_\Om \f{1}{s\la^2}e^{-\la d}(|\pa_t h|^2 + |\na u|^2)e^{2se^{\la d}}dx \\
&\hspace{2cm} + C\int_\Om se^{\la d}(|\na h|^2 + |\na u|^2 + |u|^2 + |\na H|^2 + |H|^2 + |\na \wt{p}|^2)e^{2se^{\la d}}dx
\end{align*}
for all $\la\ge \la_0$ and all $s\ge s_0$.
Since $p = q + \wt{p}$, we have
\begin{equation}
\label{eq:prf-ce-1n}
\begin{aligned}
\int_\Om \big( |\na p|^2 &+ s^2\la^2e^{2\la d}|p|^2\big) e^{2se^{\la d}}dx \\
\le \; &2\int_\Om \big( |\na q|^2 + s^2\la^2e^{2\la d}|q|^2\big) e^{2se^{\la d}}dx + 2\int_\Om \big( |\na \wt{p}|^2 + s^2\la^2e^{2\la d}|\wt{p}|^2\big) e^{2se^{\la d}}dx \\
\le \; &C\int_\Om se^{\la d}(|F|^2 + |\na h|^2)e^{2se^{\la d}}dx + Cs^2\la^2e^{2\la \|d\|_{C(\ov{\Om})}}e^{2se^{\la \|d\|_{C(\ov{\Om})}}}\|p\|_{H^{\frac{1}{2}}(\pa\Om)}^2 \\
&+ C\int_\Om \f{1}{s\la^2}e^{-\la d}|\pa_t h|^2e^{2se^{\la d}}dx + C\int_\Om se^{\la d}(|\na u|^2 + |u|^2 + |\na H|^2 + |H|^2)e^{2se^{\la d}}dx \\
\end{aligned}
\end{equation}
for all $\la\ge \la_0$ and all $s\ge s_0$. We used (\ref{eq:Stt1n}) in the last inequality.

Recall the definition of weight function \eqref{con:choice-phi-psi}. Since $se^{\la(-\beta(t-t_0)^2+c_0)}\ge s$ for all $0\le t\le T$. Hence substituting $s$ by $se^{\la(-\beta(t-t_0)^2+c_0)}$ in (\ref{eq:prf-ce-1n}) yields
\begin{align}
\nonumber
\int_\Om \big( |\na p|^2 &+ s^2\la^2\vp^2|p|^2\big) e^{2s\vp}dx \le \; C\int_\Om s\vp(|F|^2 + |\na h|^2)e^{2s\vp}dx + C\int_\Om \f{1}{s\la^2\vp}|\pa_t h|^2e^{2s\vp}dx \\
\label{eq:Ce-reg-p}
&+ C\int_\Om s\vp(|\na u|^2 + |u|^2 + |\na H|^2 + |H|^2)e^{2s\vp}dx + C(\la)s^2\e^{C(\la)s}\|p\|_{H^{\frac{1}{2}}(\pa\Om)}^2
\end{align}
for all $\la \ge \la_0$ and all $s\ge s_0$. 
\vspace{0.2cm}

\noindent $\mathbf{Second\ step.}$ We apply a Carleman estimate for parabolic type.

We have the following lemma.
\begin{lem}
\label{lem:Ce-reg-para}
Let $\vp$ be chosen as (\ref{con:choice-phi-psi}) and $y\in H^{2,1}(Q)$ satisfy
$$
\left\{
\begin{aligned}
& \quad \pa_t y - \nu(x,t)\De y + \sum_{j=1}^3 b_j(x,t)\pa_j y + c(x,t)y = f &\quad in \; Q \\
& \quad y(\cdot, 0) = y(\cdot, T) = 0 &\quad in \; \Om 
\end{aligned}
\right.
$$
with $\nu,b_j,c\in W^{1,\infty}(Q)$, $\nu\ge c_0>0$ and $f\in L^2(Q)$, $j=1,2,3$. Then there exist  constants $\la_0>0$, $s_0>0$ and $C>0$ such that
\begin{equation}
\begin{aligned}
\int_Q \bigg\{ \frac{1}{s\vp}\bigg( |\pa_t y|^2 + \sum_{i,j=1}^3 &|\pa_i \pa_j y|^2 \bigg) + s\la^2\vp|\na y|^2 + s^3\la^4\vp^3 |y|^2\bigg\} e^{2s\vp}dxdt \\
\le \; &C\int_Q |f|^2e^{2s\vp}dxdt + C(\la)e^{C(\la)s}\int_{\Sigma} (|y|^2 + |\na_{x,t} y|^2)dSdt
\end{aligned}
\end{equation}
for all $\la\ge \hat{\la}$ and all $s\ge \hat{s}$.
\end{lem}
The proof is almost the same to Therorem 3.2 in Yamamoto \cite{Y09}. 
\vspace{0.2cm}

We rewrite the first equation in (\ref{sy:MHD4n}) to get
\begin{equation*}
\pa_t u - \nu \Delta u + (B^{(1)}\cdot\na)u + (u\cdot\na)B^{(2)} + \na(B^{(3)}\cdot u)= F - \na p - L_1(H).
\end{equation*}
Applying Lemma \ref{lem:Ce-reg-para} to each component of above equations, we obtain
\begin{equation}
\label{eq:Ce-reg-u}
\begin{aligned}
&\int_Q \bigg\{ \frac{1}{s\vp}\bigg( |\pa_t u|^2 + \sum_{i,j=1}^3 |\pa_i \pa_j u|^2 \bigg) + s\la^2\vp|\na u|^2 + s^3\la^4\vp^3 |u|^2\bigg\} e^{2s\vp}dxdt \le C\int_Q |F|^2e^{2s\vp}dxdt \\
&\hspace{1cm} + C\int_Q (|\na p|^2 + |\na H|^2 + |H|^2)e^{2s\vp}dxdt + C(\la)e^{C(\la)s}\big(\|u\|_{L^2(\Sigma)}^2 + \|\na_{x,t} u\|_{L^2(\Sigma)}^2\big)
\end{aligned}
\end{equation}
for all $\la\ge \la_1$ and all $s\ge s_1$.

Next, we apply Carleman estimate of parabolic type to the second equation of (\ref{sy:MHD4n}) and we have the following estimate:
\begin{equation}
\label{eq:Ce-reg-H}
\begin{aligned}
&\int_Q \bigg\{ \frac{1}{s\vp}\bigg( |\pa_t H|^2 + \sum_{i,j=1}^3 |\pa_i \pa_j H|^2 \bigg) + s\la^2\vp |\na H|^2 + s^3\la^4\vp^3 |H|^2\bigg\} e^{2s\vp}dxdt \le C\int_Q |G|^2e^{2s\vp}dxdt \\
&\hspace{1cm} + C\int_Q (|\na u|^2 + |u|^2)e^{2s\vp}dxdt + C(\la)e^{C(\la)s}\big(\|H\|_{L^2(\Sigma)}^2 + \|\na_{x,t} H\|_{L^2(\Sigma)}^2\big)
\end{aligned}
\end{equation}
for all $\la\ge \la_2$ and all $s\ge s_2$. 
\vspace{0.2cm}

\noindent $\mathbf{Third\ step.}$ We combine the estimates for $p$, $u$ and $H$.

Combining (\ref{eq:Ce-reg-p}), (\ref{eq:Ce-reg-u}) and (\ref{eq:Ce-reg-H}), we obtain
$$
\begin{aligned}
& \int_Q \bigg\{ \frac{1}{s\vp}\bigg( |\pa_t u|^2 + \sum_{i,j=1}^3 |\pa_i \pa_j u|^2\bigg) + s\la^2\vp|\na u|^2 + s^3\la^4\vp^3|u|^2 + |\na p|^2 + s^2\la^2\vp^2 |p|^2 \\
&\hspace{1cm} + \frac{1}{s\vp}\bigg( |\pa_t H|^2 + \sum_{i,j=1}^3 |\pa_i \pa_j H|^2\bigg) + s\la^2\vp |\na H|^2 + s^3\la^4\vp^3 |H|^2 \bigg\}e^{2s\vp}dxdt \\
& \le \;C\int_Q (s\vp (|F|^2 + |\na h|^2)+ |G|^2 + \f{1}{s\vp}|\pa_t h|^2) e^{2s\vp}dxdt + C\int_Q s\vp\big( |\na u|^2 + |u|^2 + |\na H|^2 + |H|^2 \big) e^{2s\vp}dxdt \\
&\hspace{1cm} + C(\la)s^2e^{C(\la)s}\bigg(\|u\|_{L^2(\Sigma)}^2 + \|\na_{x,t} u\|_{L^2(\Sigma)}^2 + \|H\|_{L^2(\Sigma)}^2 + \|\na_{x,t} H\|_{L^2(\Sigma)}^2 + \|p\|_{L^2(0,T;H^{\frac{1}{2}}(\pa\Om))}^2 \bigg)
\end{aligned}
$$
for all $\la\ge \la_3:=\max\{\la_1,\la_2,\la_3\}$ and all $s\ge s_3:=\max\{s_0,s_1,s_2\}$. Finally we can fix $\la$ large enough to absorb the second term on the right-hand side into the left-hand side. By the relations $\la \ge 1$ and $s^2\le e^{Cs}$ for $s$ large, we obtain
\begin{align*}
&\|(u,p,H)\|_{\sigma_s(Q)}^2 \le \;C\int_Q (s\vp |F|^2 + s\vp|\na h|^2 + |G|^2 + \f{1}{s\vp}|\pa_t h|^2) e^{2s\vp}dxdt \\
&\hspace{1cm} + Ce^{Cs}\bigg(\|u\|_{L^2(\Sigma)}^2 + \|\na_{x,t} u\|_{L^2(\Sigma)}^2 + \|H\|_{L^2(\Sigma)}^2 + \|\na_{x,t} H\|_{L^2(\Sigma)}^2 + \|p\|_{L^2(0,T;H^{\frac{1}{2}}(\pa\Om))}^2 \bigg)
\end{align*}
for fixed $\la$ large enough and all $s \ge s_3$.
\vspace{0.2cm}

The proof of Theorem \ref{thm:CEDPn} is completed.
\end{proof}

On the other hand, we investigate the following two first-order partial differential operators: 
\vspace{0.2cm}

($\romannumeral1)\quad P f := \mathrm{div} (fA)=A \na f + f\mathrm{div} A, \quad f\in H^1(\Om)$, 
\vspace{0.2cm}

($\romannumeral2)\quad Q g := \mathrm{rot} (gb)=\na g\times b + g\mathrm{rot} b, \quad g\in H^1(\Om)$
\vspace{0.2cm}

\noindent where  $A=(A_{ij})_{i,j}$ is a $3\times 3$ matrix and $b=(b_1,b_2,b_3)^T$ is a vector satisfying $A\in W^{1,\infty}(\Om), b\in W^{2,\infty}(\Om)$. Recall that the divergence of a matrix is defined as $[\mathrm{div} A]_k=\sum_{j=1}^3 \pa_j A_{kj}$. In addition, we select an open subset $O\subset \Om$. 
Then we have the following Carleman inequalities: 
\begin{thm}
\label{thm:CEIPn}
Let $d$ be chosen as (\ref{con:d}) and $\vp_0:=e^{\la d}$. Assume that 
$$
\mathrm{det} A(x) \neq 0 \ and \ |\na d(x) \times b(x)|\neq 0, \qquad for\ x\in \ov{O}.
$$
Then there exist constants $\la_0\ge 1$, $s\ge 1$ and a generic constant $C>0$ such that
\begin{equation}
\label{CEIP1n}
\int_{O} (|\na f|^2 + s^2\la^2\vp_0^2 |f|^2)e^{2s\vp_0} dx \le C\int_{O} |P f|^2 e^{2s\vp_0}dx + C\int_{\pa O} s\la\vp_0 |f|^2e^{2s\vp_0}d\sigma
\end{equation}
and
\begin{equation}
\label{CEIP2n}
\begin{aligned}
&\int_{O} (|\na g|^2 + s^2\la^2\vp_0^2 |g|^2)e^{2s\vp_0} dx \le C\int_{O} (\frac{1}{s^2\la^2\vp_0^2}|\na (Qg)|^2 + |Qg|^2) e^{2s\vp_0}dx \\
&\hspace{5cm} + C\int_{\pa O} (\frac{1}{s\la\vp_0}|\na g|^2 + s\la\vp_0 |g|^2)e^{2s\vp_0}d\sigma
\end{aligned}
\end{equation}
for all $\la\ge \la_0$, $s\ge s_0$ and $f\in H^1(\Om), g\in H^2(\Om)$. 
\end{thm}
To proof these inequalities, we apply the idea of Lemma 6.1 in \cite{Y09}. 
\begin{proof}
We first prove inequality (\ref{CEIP1n}). Set $w=fe^{s\vp_0}$. Then 
$$
Pf=P(we^{-s\vp_0})=e^{-s\vp_0}(A\na w + w\mathrm{div} A - s\la\vp_0 (A\na d) w).
$$
We rewrite it in components, that is
\begin{equation}
\label{eq1n}
[Pf]_ke^{s\vp_0}=\sum_{j=1}^3 \big( A_{kj}\pa_j w + \pa_j A_{kj} w - s\la\vp_0 (A_{kj}\pa_j d)w\big)
\end{equation}
Now choose $a=(a_1,a_2,a_3)^T\in L^\infty (\Om)$ such that $\sum_{k=1}^3 a_k A_{kj} = \pa_j d$ for any $x\in \ov{O}$. In fact, the existence of such $\{a_k\}_{k=1,2,3}$ comes from the assumption $\mathrm{det} A\neq 0$ on $\ov{O}$. 

We multiply $a_k$ to equation (\ref{eq1}) and take summation over $k$:
$$
\sum_{k=1}^3 a_k[Pf]_k e^{s\vp_0} = \na d\cdot \na w + \big(\sum_{j,k=1}^3 a_k \pa_j A_{kj}\big) w - s\la\vp_0 |\na d|^2 w \qquad on \ \ov{O}. 
$$
Then we estimate
\begin{align*}
&\int_{O} \Big|\sum_{k=1}^3 a_k[Pf]_k \Big|^2e^{2s\vp_0}dx = \int_O s^2\la^2\vp_0^2|\na d|^4|w|^2 dx + \int_O |\na d \cdot \na w + (a\cdot \mathrm{div} A) w|^2 dx \\
&\hspace{4.3cm} - 2\int_O s\la\vp_0 |\na d|^2(\na d\cdot \na w + (a\cdot\mathrm{div}A) w)w dx \\
&\hspace{3.9cm}\ge \int_O s^2\la^2\vp_0^2|\na d|^4|w|^2 dx - 2\int_O s\la\vp_0 |\na d|^2 (a\cdot \mathrm{div} A)|w|^2 dx \\
&\hspace{4.3cm} - \int_{\pa O} s\la\vp_0 |\na d|^2 \frac{\pa d}{\pa n}|w|^2 d\sigma + \int_O s\la \mathrm{div} (\vp_0|\na d|^2 \na d)|w|^2 dx \\
&\hspace{3.9cm}\ge \int_O s^2\la^2\vp_0^2|\na d|^4|w|^2 dx - \int_{\pa O} s\la\vp_0 |\na d|^2 \frac{\pa d}{\pa n}|w|^2 d\sigma \\
&\hspace{4.3cm} + \int_O s\la\vp_0 \big(\la|\na d|^4 + \na |\na d|^2 \na d + |\na d|^2 (\Delta d - 2(a\cdot\mathrm{div} A))\big)|w|^2 dx. 
\end{align*}
By choosing $\la$ large, we can absorb the third term on the right-hand side. Thus,
\begin{equation}
\label{eq2n}
\int_O s^2\la^2\vp_0^2|f|^2e^{2s\vp_0} dx \le C \int_O |Pf|^2e^{2s\vp_0}dx + C \int_{\pa O} s\la\vp_0|f|^2e^{2s\vp_0} d\sigma
\end{equation} 
holds for all $\la\ge \la_1$ and $s\ge 1$. 

Furthermore, for $l=1,2,3$, we can also choose $a^{(l)}=(a_1^{(l)},a_2^{(l)},a_3^{(l)})^T\in L^\infty (\Om)$ such that $\sum_{k=1}^3 a_k^{(l)} A_{kj}=\delta_{lj}$ on $\ov{O}$. Take summation over $k$ after multiply $a_k^{(l)}$ to (\ref{eq1n}):
$$
\sum_{k=1}^3 a_k^{(l)}[Pf]_k e^{s\vp_0} = \pa_l w + \big(\sum_{j,k=1}^3 a_k^{(l)} \pa_j A_{kj}\big) w - s\la\vp_0 (\pa_l d) w \qquad on \ \ov{O}.
$$
Again we estimate
\begin{align*}
&\int_O \Big|\sum_{k=1}^3 a_k^{(l)}[Pf]_k \Big|^2e^{2s\vp_0}dx = \int_O |\pa_l w|^2 dx + \int_O |(a^{(l)}\cdot \mathrm{div} A) - s\la\vp_0(\pa_l d)|^2 |w|^2 dx \\
&\hspace{4.3cm} + 2\int_O \big((a^{(l)}\cdot \mathrm{div} A) - s\la\vp_0(\pa_l d)\big)w(\pa_l w) dx \\
&\hspace{3.9cm}\ge \int_O |\pa_l w|^2 dx + 2\int_O (a^{(l)}\cdot \mathrm{div} A)w(\pa_l w) dx \\
&\hspace{4.3cm} - \int_{\pa O} s\la\vp_0 (\pa_l d)n_l |w|^2 d\sigma + \int_O s\la\vp_0(\la|\pa_l d|^2 + \pa_l^2 d)|w|^2 dx.
\end{align*}
Rewrite the above inequality and take summation over $l$ on both sides:
\begin{align*}
&\int_O |\na w|^2 dx \le \int_O \sum_{l=1}^3\Big|\sum_{k=1}^3 a_k^{(l)}[Pf]_k \Big|^2e^{2s\vp_0}dx + \int_{\pa O} s\la\vp_0 \frac{\pa d}{\pa n} |w|^2 d\sigma \\
&\hspace{2.5cm} - 2\sum_{l=1}^3\int_O (a^{(l)}\cdot \mathrm{div} A)w(\pa_l w) dx - \int_O s\la\vp_0(\la|\na d|^2 + \Delta d)|w|^2 dx \\
&\hspace{1.8cm} \le C\int_O |Pf|^2e^{2s\vp_0}dx + \int_{\pa O} s\la\vp_0\frac{\pa d}{\pa n} |w|^2 d\sigma \\
&\hspace{2.5cm} + \frac{1}{2}\int_O |\na w|^2 dx + 2\int_O \sum_{l=1}^3 |a^{(l)}\cdot\mathrm{div} A|^2 |w|^2 dx + \int_O s\la\vp_0|w|^2 dx
\end{align*}
This leads to
$$
\int_O |\na w|^2 dx \le C\int_O |Pf|^2e^{2s\vp_0}dx + C\int_{\pa O} s\la\vp_0 |w|^2 d\sigma + C\int_O s\la\vp_0|w|^2 dx
$$
Together with (\ref{eq2n}) and take $\la$ large enough to absorb the last term on the right-hand side. Finally, we obtain
$$
\int_{O} (|\na f|^2 + s^2\la^2\vp_0^2 |f|^2)e^{2s\vp_0} dx \le C\int_{O} |P f|^2 e^{2s\vp_0}dx + C\int_{\pa O} s\la\vp_0 |f|^2e^{2s\vp_0}d\sigma
$$
for all $\la\ge \la_2$ and $s\ge 1$. 

Next we consider the operator $Q$. 
Set $v=ge^{s\vp_0}$. Then
$$
Qg=Q(ve^{-s\vp_0})=e^{-s\vp_0}(\na v\times b + (\mathrm{rot} b)v - s\la\vp_0(\na d\times b)v).
$$
By denoting
$$
B=
\begin{pmatrix}
0 & b_3 & -b_2 \\
-b_3 & 0 & b_1 \\
b_2 & -b_1 & 0 \\
\end{pmatrix},
$$
we rewrite the above formula:
$$
Qg e^{s\vp_0} = B\na v + (\mathrm{rot} b)v - s\la\vp_0(B\na d)v.
$$
However, $\mathrm{det} B = b_1b_2b_3 + (-b_1b_2b_3)=0$. Thus, we calculate directly
\begin{align*}
&\int_O |Qg|^2e^{2s\vp_0} dx = \int_O s^2\la^2\vp_0^2 |B\na d|^2 |v|^2 dx + \int_O |B\na v + (\mathrm{rot} b)v|^2 dx \\
&\hspace{2.7cm} -2\int_O s\la\vp_0 (B\na d)\cdot (B\na v + (\mathrm{rot} b)v)v dx \\
&\hspace{2.5cm} \ge \int_O s^2\la^2\vp_0^2 |B\na d|^2 |v|^2 dx -2\int_O s\la\vp_0(B\na d)\cdot (\mathrm{rot} b)|v|^2 dx \\
&\hspace{2.7cm} -\int_{\pa O} s\la\vp_0 (B\na d)\cdot (Bn)|v|^2 d\sigma + \int_O s\la\vp_0 \big(\la |B\na d|^2 + \mathrm{div}(B^T(B\na d))\big) |v|^2 dx
\end{align*}
By noting the assumption that $|B\na d|=|\na d\times b| \neq 0$ on $\ov{O}$, we can take $\la$ large to absorb the second and fourth terms on the right-hand side:
\begin{equation}
\label{eq3n}
\int_O s^2\la^2\vp_0^2 |g|^2e^{2s\vp_0} dx \le C\int_O |Qg|^2e^{2s\vp_0} dx + C\int_{\pa O}s\la\vp_0 |g|^2e^{2s\vp_0} d\sigma
\end{equation}
for all $\la\ge \la_3$ and $s\ge 1$. 

We take the $k$-th derivative of (\romannumeral2) and denote $g_k=\pa_k g$. Set
$$
Q_k g_k:= \pa_k(Qg) - \na g \times \pa_k b - g(\mathrm{rot}(\pa_k b))= \na g_k\times b + g_k(\mathrm{rot}b).
$$
By applying similar argument above to operator $Q_k$, we have 
\begin{align*}
&\int_O |g_k|^2e^{2s\vp_0} dx \le C\int_O \frac{1}{s^2\la^2\vp_0^2}|Q_k g_k|^2e^{2s\vp_0} dx + C\int_{\pa O} \frac{1}{s\la\vp_0} |g_k|^2e^{2s\vp_0} d\sigma \\
&\hspace{0cm} \le C\int_O \frac{1}{s^2\la^2\vp_0^2}|\pa_k (Qg)|^2e^{2s\vp_0} dx + C\int_O \frac{1}{s^2\la^2\vp_0^2}(|\na g|^2 +|g|^2)e^{2s\vp_0}dx + C\int_{\pa O}\frac{1}{s\la\vp_0} |g_k|^2e^{2s\vp_0} d\sigma
\end{align*}
for all $\la \ge \la_4$, $s\ge 1$ and $k=1,2,3$. Sum up the estimates over $k$ and absorb again the lower-order terms by taking $\la$ large:
\begin{equation}
\label{eq4n}
\begin{aligned}
&\int_O |\na g|^2e^{2s\vp_0} dx \le C\int_O \frac{1}{s^2\la^2\vp_0^2}|\na (Qg)|^2e^{2s\vp_0} dx + C\int_O \frac{1}{s^2\la^2\vp_0^2}|g|^2e^{2s\vp_0}dx \\
&\hspace{2.8cm}+ C\int_{\pa O} \frac{1}{s\la\vp_0} |\na g|^2e^{2s\vp_0} d\sigma
\end{aligned}
\end{equation}
for all $\la \ge \la_5$ and all $s\ge 1$. Combining (\ref{eq3n}) and (\ref{eq4n}), we proved (\ref{CEIP2n}) and also Theorem \ref{thm:CEIPn} with $\la_0=max\{\la_i:\ 1\le i\le 5\}$ and $s_0=1$.   

\end{proof}
Recall that our regular weight function is defined as
$$
\vp(x,t) = e^{\la\psi(x,t)}, \quad \psi(x,t) = d(x) - \beta(t-t_0)^2 +c_0.
$$
For all $s\ge s_0$, $se^{\la c_0}\ge s\ge s_0$. Then substituting $s$ by $se^{\la c_0}$ in \eqref{CEIP1n} and \eqref{CEIP2n} leads to
\begin{thm}
\label{thm:CEIP-reg}
Under the assumptions that 
$$
\mathrm{det} A(x) \neq 0 \ and \ |\na d(x) \times b(x)|\neq 0, \qquad for\ x\in \ov{O},
$$
there exist constants $\la_0\ge 1$, $s_0\ge 1$ and a generic constant $C>0$ such that
\begin{align*}
\int_{O} (|\na f|^2 + s^2\la^2\vp^2(x,t_0) |f|^2)e^{2s\vp(x,t_0)} dx \le C\int_{O} |P f|^2 e^{2s\vp(x,t_0)}dx + C\int_{\pa O} s\la\vp_0 |f|^2e^{2s\vp(x,t_0)}d\sigma
\end{align*}
and
\begin{align*}
&\int_{O} (|\na g|^2 + s^2\la^2\vp^2(x,t_0) |g|^2)e^{2s\vp(x,t_0)} dx \le C\int_{O} (\frac{1}{s^2\la^2\vp^2(x,t_0)}|\na (Qg)|^2 + |Qg|^2) e^{2s\vp(x,t_0)}dx \\
&\hspace{5cm} + C\int_{\pa O} (\frac{1}{s\la\vp(x,t_0)}|\na g|^2 + s\la\vp(x,t_0) |g|^2)e^{2s\vp(x,t_0)}d\sigma
\end{align*}
for all $\la\ge \la_0$, $s\ge s_0$ and $f\in H^1(\Om), g\in H^2(\Om)$. 
\end{thm}


\section{Proof of Theorem \ref{thm:stability} and \ref{thm:stabilityn}}
In this section, we prove our stability result(Theorem \ref{thm:stability},  \ref{thm:stabilityn}) in terms of the two types of Carleman inequalities established in the last section. 

First of all, we change our inverse coefficients problem to an inverse source problem. Recall that we have two sets of solutions $(u_i,p_i,H_i)$(i=1,2) satisfying the following MHD system:
\begin{equation}
\label{sy:MHD5}
\left\{
\begin{aligned}
&\pa_t u_i - \mathrm{div}(2\nu_i\mathcal{E}(u_i)) + (u_i\cdot\nabla) u_i - (H_i\cdot\nabla) H_i + \na p_i - \nabla H_i^T\!\cdot\! H_i = 0 &\quad in \ Q, \\
&\pa_t H_i + \mathrm{rot}(\kappa_i\mathrm{rot}\; H_i) + (u_i\cdot\nabla) H_i - (H_i\cdot\nabla) u_i = 0 &\quad in \ Q, \\
& \mathrm{div}\; u_i = 0, \quad \mathrm{div}\; H_i = 0 &\quad in \ Q.
\end{aligned}
\right.
\end{equation}
Take the difference of the two sets of equations in (\ref{sy:MHD5}). By setting $u = u_1 - u_2$, $H = H_1 - H_2$, $p = p_1 - p_2$ and $\nu = \nu_1 - \nu_2$, $\kappa = \kappa_1 - \kappa_2$, we obtain
\begin{equation}
\label{sy:MHD6}
\left\{
\begin{aligned}
&\pa_t u - \nu_2\Delta u + (u\!\cdot\!\nabla) u_2 + ((u_1-\!\nabla \nu_2)\!\cdot\!\nabla) u - \nabla u^T\!\cdot\!\nabla\nu_2 + L_1(H,\nabla H) + \na p = \mathrm{div}(2\nu\mathcal{E}(u_1)), \\
&\pa_t H - \kappa_2\Delta H - (H\!\cdot\!\nabla) u_2 + (u_1\!\cdot\!\nabla) H + \nabla\kappa_2\times\mathrm{rot}H + L_2(u,\nabla u)  = -\mathrm{rot}(\kappa\mathrm{rot}H_1), \\
& \mathrm{div}\; u = 0, \quad \mathrm{div}\; H = 0, \hspace{9cm} in\ Q
\end{aligned}
\right.
\end{equation}
Here
\begin{align*}
&L_1(H,\nabla H) = - (H_1\cdot\na) H - (H\cdot\na) H_2 - \na H^T\cdot H_2 - \na H_1^T\cdot H, \\
&L_2(u,\nabla u) = - (H_1\cdot\na) u + (u\cdot\na) H_2
\end{align*}

\subsection{Proof of Theorem \ref{thm:stability}}
Note that $t_0\in (0,T)$ is the fixed time for measurements. By the assumptions (A1)-(A2), we can replace coefficients $A$ and $b$ in Theorem \ref{thm:CEIP-sin} by $2\mathcal{E}(u_1(\cdot,t_0))$ and $\mathrm{rot} H_1(\cdot,t_0)$. This leads to
\begin{align}
\label{eq5}
\int_{\Om} (|\na \nu|^2 + s^2\la^2\vp^2(x,t_0) |\nu|^2)e^{2s\al (x, t_0)} dx \le C\int_{\Om} |\mathrm{div}(2\nu\mathcal{E}(u_1))(x,t_0)|^2 e^{2s\al (x,t_0)}dx
\end{align}
and 
\begin{align}
\nonumber
&\int_{\Om} (|\na \kappa|^2 + s^2\la^2\vp^2(x,t_0) |\kappa|^2)e^{2s\al (x,t_0)} dx \le C\int_{\Om} |\mathrm{rot}(\kappa\mathrm{rot}H_1)(x,t_0)|^2 e^{2s\al (x,t_0)}dx \\
\label{eq6}
&\hspace{4cm} + C\int_{\Om} \frac{1}{s^2\la^2\vp^2(x,t_0)}|\na (\mathrm{rot}(\kappa\mathrm{rot}H_1))|^2 e^{2s\al (x,t_0)}dx 
\end{align}
for all $\la\ge \la_0$ and all $s\ge s_0$. Henceforth, we may omit $t_0$ when there is no confusion. We multiply by $s$ on both sides of \eqref{eq5} and \eqref{eq6} and then take the summation:
\begin{equation}
\label{eq7}
\begin{aligned}
&\int_{\Om} s(|\na \nu|^2 + s^2\la^2\vp^2 |\nu|^2 + |\na \kappa|^2 + s^2\la^2\vp^2 |\kappa|^2)e^{2s\al} dx \\
&\hspace{1cm} \le C\int_{\Om} s\Big(|\mathrm{div}(2\nu\mathcal{E}(u_1))|^2 + |\mathrm{rot}(\kappa\mathrm{rot}H_1)|^2 + \frac{1}{s^2\la^2\vp^2}|\na (\mathrm{rot}(\kappa\mathrm{rot}H_1))|^2\Big) e^{2s\al}dx
\end{aligned} 
\end{equation}
holds for all $\la\ge \la_0$, $s\ge s_0$ and $t=t_0$. 
\begin{align*}
&[the\ RHS\ of\ (\ref{eq7})] \le Cs\int_\Om (|\pa_t u|^2 + |\Delta u|^2 + |\na u|^2 + |L_1(H,\na H)|^2 +|\na p|^2 )e^{2s\al}dx \\
&\hspace{2cm} + Cs\int_\Om \frac{1}{s^2\la^2\vp^2}( |\na (\pa_t H)|^2 + |\na(\Delta H)|^2 + |\pa_i\pa_j H|^2 + |\na (L_2(u,\na u))|^2)e^{2s\al}dx \\
&\hspace{2cm} + Cs\int_\Om (|\pa_t H|^2 + |\Delta H|^2 + |\na H|^2 + |L_2(u,\na u)|^2 )e^{2s\al}dx \\
&\hspace{2.8cm} \le C\int_\Om s(|\pa_t u|^2 + |\pa_t H|^2 + \frac{1}{s^2\la^2\vp^2}|\na (\pa_t H)|^2)e^{2s\al}dx + Cs\mathcal{D}_1^2
\end{align*}
where 
$$
\mathcal{D}_1:= \|u(\cdot,t_0)\|_{H^2(\Om)} + \|H(\cdot,t_0)\|_{H^3(\Om)} + \|\na p(\cdot,t_0)\|_{L^2(\Om)}.
$$
Next, we use Theorem \ref{thm:CEDP}(Carleman estimate for direct problem) to estimate the first integral on the right-hand side. Notice that $e^{2s\al(x,0)}=0,\ x\in\Om$, we calculate
\begin{equation}
\label{eq8}
\begin{aligned}
&\int_\Om s|\pa_t u(x,t_0)|^2e^{2s\al(x,t_0)}dx = \int_0^{t_0} \f{\pa}{\pa t}\bigg( \int_\Om s|\pa_t u|^2e^{2s\al}dx\bigg) dt \\
&\hspace{3.7cm}= \int_\Om\int_0^{t_0} \bigg( 2s(\pa_t u\cdot\pa_t^2 u) + 2s^2(\pa_t\al)|\pa_t u|^2 \bigg) e^{2s\al}dxdt \\
&\hspace{3.7cm}\le C(\la)\int_Q (|\pa_t^2 u|^2 + s^2\vp^2|\pa_t u|^2)e^{2s\al}dxdt.
\end{aligned}
\end{equation}
We used $s\ge 1$ and
$$
\begin{aligned}
&2s^2|\pa_t \al| = 2s^2\left|\f{l^\prime}{l^2}(e^{\la\eta}-e^{2\la})\right|\le C(\la)s^2\vp^2, \\
&2s |\pa_t u\cdot\pa_t^2 u| \le \f{1}{\vp^2}|\pa_t^2 u|^2 + s^2\vp^2|\pa_t u|^2 \le l^2(t_0)|\pa_t^2 u|^2 + s^2\vp^2|\pa_t u|^2.
\end{aligned}
$$
Similarly, we have
\begin{equation}
\label{eq9}
\begin{aligned}
&\int_\Om s|\pa_t H(x,t_0)|^2e^{2s\al(x,t_0)}dx = \int_0^{t_0} \f{\pa}{\pa t}\bigg( \int_\Om s |\pa_t H|^2e^{2s\al}dx\bigg) dt \\
&\hspace{3.7cm}= \int_\Om\int_0^{t_0} \bigg( 2s(\pa_t H\cdot\pa_t^2 H) + 2s^2(\pa_t\al)|\pa_t H|^2 \bigg) e^{2s\al}dxdt \\
&\hspace{3.7cm}\le C(\la)\int_Q (|\pa_t^2 H|^2 + s^2\vp^2|\pa_t H|^2)e^{2s\al}dxdt.
\end{aligned}
\end{equation}
and
\begin{equation}
\label{eq10}
\begin{aligned}
&s\int_\Om \frac{1}{s^2\la^2\vp(x,t_0)^2}|\na(\pa_t H(x,t_0))|^2e^{2s\al(x,t_0)}dx \\
&\hspace{1cm}= \int_0^{t_0} \f{\pa}{\pa t}\bigg( \int_\Om \frac{1}{s\la^2\vp^2}|\na(\pa_t H)|^2e^{2s\al}dx\bigg) dt \\
&\hspace{1cm}= \int_\Om\int_0^{t_0} \bigg( \frac{2}{s\la^2\vp^2}(\na(\pa_t H):\na(\pa_t^2 H)) + \frac{1}{s\la^2\vp^2}2s(\pa_t\al) |\na(\pa_t H)|^2 \bigg) e^{2s\al} dxdt \\
&\hspace{1cm}\le C(\la)\int_Q (\frac{1}{s^2\vp^2}|\na(\pa_t^2 H)|^2 + \frac{1}{\vp^2}|\na(\pa_t H)|^2 + |\na(\pa_t H)|^2)e^{2s\al} dxdt \\
&\hspace{1cm}\le C(\la)\int_Q (\frac{1}{s^2\vp^2}|\na(\pa_t^2 H)|^2 + |\na(\pa_t H)|^2)e^{2s\al} dxdt.
\end{aligned}
\end{equation}

Set $w_1=\pa_t u$, $w_2=\pa_t^2 u$, $q_1=\pa_t p$, $q_2=\pa_t^2 p$ and $h_1=\pa_t H$, $h_2=\pa_t^2 H$. Then according to our governing system (\ref{sy:MHD6}), we have
$$
\left\{
\begin{aligned}
&\pa_t u - \nu_2\Delta u + (u\!\cdot\!\nabla) u_2 + ((u_1-\!\nabla \nu_2)\!\cdot\!\nabla) u - \nabla u^T\!\cdot\!\nabla\nu_2 + L_1(H,\nabla H) + \na p = \mathrm{div}(2\nu\mathcal{E}(u_1)), \\
&\pa_t H - \kappa_2\Delta H - (H\!\cdot\!\nabla) u_2 + (u_1\!\cdot\!\nabla) H + \nabla\kappa_2\times\mathrm{rot}H + L_2(u,\nabla u)  = -\mathrm{rot}(\kappa\mathrm{rot}H_1), \\
& \mathrm{div}\; u = 0, \quad \mathrm{div}\; H = 0
\end{aligned}
\right.
$$
and
\vspace{0.2cm}
$$
\left\{
\begin{aligned}
&\pa_t w_1 - \nu_2\Delta w_1 + (w_1\!\cdot\!\nabla) u_2 + ((u_1-\!\nabla \nu_2)\!\cdot\!\nabla) w_1 - \nabla w_1^T\!\cdot\!\nabla\nu_2 + L_1(h_1,\nabla h_1) + \na q_1 \\
&\hspace{4cm} = \mathrm{div}(2\nu\mathcal{E}((\pa_t u_1))) - (u\!\cdot\!\na)(\pa_t u_2) - ((\pa_t u_1)\!\cdot\!\nabla) u - L_{1t}(H,\na H), \\
&\pa_t h_1 - \kappa_2\Delta h_1 - (h_1\!\cdot\!\nabla) u_2 + (u_1\!\cdot\!\nabla) h_1 + \nabla\kappa_2\times\mathrm{rot}h_1 + L_2(w_1,\nabla w_1)  \\
&\hspace{4cm} = -\mathrm{rot}(\kappa\mathrm{rot}(\pa_t H_1)) + (H\!\cdot\!\nabla) (\pa_t u_2) - ((\pa_t u_1)\!\cdot\!\nabla) H - L_{2t}(u,\na u), \\
& \mathrm{div}\; w_1 = 0, \quad \mathrm{div}\; h_1 = 0
\end{aligned}
\right.
$$
and
\vspace{0.2cm}
$$
\left\{
\begin{aligned}
&\pa_t w_2 - \nu_2\Delta w_2 + (w_2\!\cdot\!\nabla) u_2 + ((u_1-\!\nabla \nu_2)\!\cdot\!\nabla) w_2 - \nabla w_2^T\!\cdot\!\nabla\nu_2 + L_1(h_2,\nabla h_2) + \na q_2 \\
&\hspace{3cm} = \mathrm{div}(2\nu\mathcal{E}((\pa_t^2 u_1))) - 2(w_1\!\cdot\!\na)(\pa_t u_2) - 2((\pa_t u_1)\!\cdot\!\nabla) w_1 - 2L_{1t}(h_1,\na h_1) \\
&\hspace{3.5cm} - (u\!\cdot\!\na)(\pa_t^2 u_2) - ((\pa_t^2 u_1)\!\cdot\!\nabla) u - L_{1tt}(u,\na u), \\
&\pa_t h_2 - \kappa_2\Delta h_2 - (h_2\!\cdot\!\nabla) u_2 + (u_1\!\cdot\!\nabla) h_2 + \nabla\kappa_2\times\mathrm{rot}h_2 + L_2(w_2,\nabla w_2)  \\
&\hspace{3cm} = -\mathrm{rot}(\kappa\mathrm{rot}(\pa_t^2 H_1)) + 2(h_1\!\cdot\!\nabla) (\pa_t u_2) - 2((\pa_t u_1)\!\cdot\!\nabla) h_1 - 2L_{2t}(w_1,\na w_1) \\
&\hspace{3.5cm} + (H\!\cdot\!\nabla) (\pa_t^2 u_2) - ((\pa_t^2 u_1)\!\cdot\!\nabla) H - L_{2tt}(u,\na u), \\
& \mathrm{div}\; w_2 = 0, \quad \mathrm{div}\; h_2 = 0.
\end{aligned}
\right.
$$

Apply Theorem \ref{thm:CEDP} to $(u,p,H)$, then to $(w_1,q_1,h_1)$ and then to $(w_2,q_2,h_2)$ respectively, we obtain
\begin{equation*}
\begin{aligned}
\|&(u,p,H)\|_{\chi_s(Q)}^2 \le \;C\int_Q (|\nu|^2 + |\na \nu|^2 + |\kappa|^2 + |\na \kappa|^2 )e^{2s\alpha}dxdt \\
&\hspace{1cm} + Ce^{-s}\bigg(\|u\|_{L^2(\Sigma)}^2 + \|\na_{x,t} u\|_{L^2(\Sigma)}^2 + \|H\|_{L^2(\Sigma)}^2 + \|\na_{x,t} H\|_{L^2(\Sigma)}^2 + \|p\|_{L^2(0,T;H^{\frac{1}{2}}(\pa \Om))}^2 \bigg)
\end{aligned}
\end{equation*}
and
\begin{equation*}
\begin{aligned}
&\|(w_1,q_1,h_1)\|_{\chi_s(Q)}^2 \le \;C\int_Q (|\nu|^2 + |\na \nu|^2 + |\kappa|^2 + |\na \kappa|^2 + |u|^2 + |\na u|^2 + |H|^2 + |\na H|^2)e^{2s\alpha}dxdt \\
&\hspace{0.5cm} + Ce^{-s}\bigg(\|w_1\|_{L^2(\Sigma)}^2 + \|\na_{x,t} w_1\|_{L^2(\Sigma)}^2 + \|h_1\|_{L^2(\Sigma)}^2 + \|\na_{x,t} h_1\|_{L^2(\Sigma)}^2 + \|q_1\|_{L^2(0,T;H^{\frac{1}{2}}(\pa \Om))}^2 \bigg)
\end{aligned}
\end{equation*}
and
\begin{equation*}
\begin{aligned}
&\|(w_2,q_2,h_2)\|_{\chi_s(Q)}^2 \le \;C\int_Q (|\nu|^2 + |\na \nu|^2 + |\kappa|^2 + |\na \kappa|^2 + |u|^2 + |\na u|^2 )e^{2s\alpha}dxdt \\
&\hspace{0.5cm} + C\int_Q (|H|^2 + |\na H|^2 + |w_1|^2 + |\na w_1|^2 + |h_1|^2 + |\na h_1|^2)e^{2s\alpha}dxdt \\
&\hspace{0.5cm} + Ce^{-s}\bigg(\|w_2\|_{L^2(\Sigma)}^2 + \|\na_{x,t} w_2\|_{L^2(\Sigma)}^2 + \|h_2\|_{L^2(\Sigma)}^2 + \|\na_{x,t} h_2\|_{L^2(\Sigma)}^2 + \|q_2\|_{L^2(0,T;H^{\frac{1}{2}}(\pa \Om))}^2 \bigg)
\end{aligned}
\end{equation*}
We combine the above three estimates and absorb the lower-order terms on the right-hand side. Then we have
\begin{equation}
\label{eq11}
\begin{aligned}
\sum_{j=0}^2\|(\pa_t^j u,\pa_t^j p, \pa_t^j H)\|_{\chi_s(Q)}^2 \le\; C\int_Q (|\nu|^2 + |\na \nu|^2 + |\kappa|^2 + |\na \kappa|^2)e^{2s\al}dxdt + Ce^{-s}\mathcal{D}_2^2
\end{aligned}
\end{equation}
for fixed $\la\ge \hat{\la}$ and all $s\ge \hat{s}$. Here
$$
\begin{aligned}
\mathcal{D}_2 = &\|u\|_{H^2(0,T;H^1(\pa\Om))} + \|u\|_{H^3(0,T;L^2(\pa\Om))} + \|\pa_n u\|_{H^2(0,T;L^2(\pa\Om))} + \|H\|_{H^2(0,T;H^1(\pa\Om))} \\
&+ \|H\|_{H^3(0,T;L^2(\pa\Om))} + \|\pa_n H\|_{H^2(0,T;L^2(\pa\Om))} + \|p\|_{H^2(0,T;H^{\frac{1}{2}}(\pa\Om))} \\
= & \|u\|_{H^{0,2}(\Sigma)} + \|\na_{x,t} u\|_{H^{0,2}(\Sigma)} + \|H\|_{H^{0,2}(\Sigma)} + \|\na_{x,t} H\|_{H^{0,2}(\Sigma)} + \|p\|_{H^{1/2,2}(\Sigma)}
\end{aligned}
$$
Fix $\la$ large($\la\ge \hat{\la}$) in inequalities (\ref{eq8})-(\ref{eq10}) and then sum them up in terms of (\ref{eq11}): 
$$
\begin{aligned}
&\int_\Om s|\pa_t u(x,t_0)|^2e^{2s\al(x)}dx + \int_\Om s|\pa_t H(x,t_0)|^2e^{2s\al(x)}dx + \int_\Om s|\na(\pa_t H(x,t_0))|^2e^{2s\al(x)}dx \\
&\hspace{4cm} \le C\int_Q (|\nu|^2 + |\na \nu|^2 + |\kappa|^2 + |\na \kappa|^2)e^{2s\al}dxdt + C\mathcal{D}_2^2.
\end{aligned}
$$
Thus, (\ref{eq7}) yields
\begin{equation}
\begin{aligned}
&\int_{\Om} s(|\na \nu|^2 + s^2\vp^2 |\nu|^2 + |\na \kappa|^2 + s^2\vp^2 |\kappa|^2)e^{2s\al} dx \\
&\hspace{1cm} \le C\int_Q (|\nu|^2 + |\na \nu|^2 + |\kappa|^2 + |\na \kappa|^2)e^{2s\al}dxdt + Cs\mathcal{D}^2.
\end{aligned} 
\end{equation}
where
$$
\mathcal{D}^2\equiv \mathcal{D}_1^2 + \mathcal{D}_2^2.
$$
We can absorb the first integral on the RHS onto the LHS for $s$ large:
\begin{align*}
&\int_{\Om} (|\na \nu|^2 + |\nu|^2 + |\na \kappa|^2 + |\kappa|^2)e^{2s\al} dx \le Cs\mathcal{D}^2.
\end{align*} 
Here we used $\al(x,t)\le\al(x,t_0)$ thanks to the choice of function $l$.

In the end, we fix $s$ sufficiently large and then the weight function $e^{2s\beta}$ admits a positive lower bound in $\Om$. This completes the proof of our main result.

\qed
\vspace{0.2cm}

\subsection{Proof of Theorem \ref{thm:stabilityn}}
Note that $t_0\in (0,T)$ is the fixed time for measurements. By the assumptions (A1$^\prime$)-(A2$^\prime$), we substitute coefficients $A$ and $b$ in Theorem \ref{thm:CEIP-reg} by $2\mathcal{E}(u_1(\cdot,t_0))$ and $\mathrm{rot} H_1(\cdot,t_0)$ so that we get Carleman type estimates. However, we cannot apply the theorem directly because we only know the information about $\nu,\kappa$ on the partial boundary $\Gamma$. Therefore we introduce level sets:
\begin{align}
\label{con:levelsets1}
\Om_\ep := \{x\in \Om: d(x)>\ep\}\quad \text{for any }\ep>0. 
\end{align} 
Then select a cut-off function $\chi_1\in C^\infty(\Bbb{R}^3)$ such that $0\le\chi_1\le 1$ and
\begin{equation*}
\chi_1 =\left\{
\begin{aligned}
&1\qquad \text{for }d > 4\ep \\
&0\qquad \text{for }d < 3\ep. 
\end{aligned}
\right.
\end{equation*}
By setting $\wt \nu = \chi_1\nu$ and $\wt \kappa = \chi_1\kappa$, we apply Theorem \ref{thm:CEIP-reg} to $\wt\nu,\wt\kappa$ with $O = \Om_{3\ep}$. Thanks to the choice of $\Om_{3\ep}$ and $\chi_1$, we have $\pa\Om_{3\ep}\cap\pa\Om\subset\Gamma$ and $\wt\nu=0 $ on $\pa\Om_{3\ep}\cap \Om$ which imply
\begin{equation}
\label{eq5n}
\int_{\Om_{3\ep}} (|\na \wt\nu|^2 + s^2\la^2\vp^2(x,t_0) |\wt\nu|^2)e^{2s\vp(x,t_0)} dx \le C\int_{\Om_{3\ep}} |\mathrm{div}(2\wt\nu\mathcal{E}(u_1))(x,t_0)|^2 e^{2s\vp(x,t_0)}dx.
\end{equation}
Also we derive
\begin{align}
\nonumber
&\int_{\Om_{3\ep}} (|\na \wt\kappa|^2 + s^2\la^2\vp^2(x,t_0) |\wt\kappa|^2)e^{2s\vp(x,t_0)} dx \le C\int_{\Om_{3\ep}} |\mathrm{rot}(\wt\kappa\mathrm{rot}H_1)(x,t_0)|^2 e^{2s\vp(x,t_0)}dx \\
\label{eq6n}
&\hspace{4cm} + C\int_{\Om_{3\ep}} \frac{1}{s^2\la^2\vp^2(x,t_0)}|\na (\mathrm{rot}(\wt\kappa\mathrm{rot}H_1))(x,t_0)|^2 e^{2s\vp(x,t_0)}dx.
\end{align}
Here the boundary integrals vanished since we have condition \eqref{con:bdyn}. Direct calculations lead to the equations
$$
\mathrm{div}(2\wt\nu\mathcal{E}(u_1)) = \chi_1 \mathrm{div}(2\nu\mathcal{E}(u_1)) + \nu \mathcal{E}(u_1)\na \chi_1,\quad \na\wt\nu = \chi_1\na\nu + \nu\na\chi_1
$$
and 
$$
\mathrm{rot}(\wt\kappa\mathrm{rot}H_1) = \chi_1 \mathrm{rot}(\kappa\mathrm{rot}H_1) + \kappa \na\chi_1\times\mathrm{rot} H_1,\quad \na\wt\kappa = \chi_1\na\kappa + \kappa\na\chi_1
$$
which together with \eqref{eq5n} and \eqref{eq6n} imply
\begin{align*}
&\int_{\Om_{4\ep}} (|\na\nu|^2 + s^2\la^2\vp^2 |\nu|^2)e^{2s\vp} dx \le C\int_{\Om_{3\ep}} |\mathrm{div}(2\nu\mathcal{E}(u_1))|^2 e^{2s\vp}dx + C\int_{\Om_{3\ep}\setminus \ov{\Om_{4\ep}}} |\nu|^2 e^{2s\vp}dx 
\end{align*}
and
\begin{align*} 
&\int_{\Om_{4\ep}} (|\na\kappa|^2 + s^2\la^2\vp^2 |\kappa|^2)e^{2s\vp} dx \le C\int_{\Om_{3\ep}} \Big( |\mathrm{rot}(\kappa\mathrm{rot}H_1)|^2 + \frac{1}{s^2\la^2\vp^2}|\na (\mathrm{rot}(\kappa\mathrm{rot}H_1))|^2 \Big) e^{2s\vp}dx \\
&\hspace{5cm} + C\int_{\Om_{3\ep}\setminus \ov{\Om_{4\ep}}}\Big(|\kappa|^2 + \frac{1}{s^2\la^2\vp^2}|\na\kappa|^2\Big)e^{2s\vp}dx
\end{align*}
for all $\la\ge \la_0$, $s\ge s_0$ and $t=t_0$. Here and henceforth we may omit $t_0$ in the estimates while we exactly mean that the estimates hold for $t=t_0$. The domain of the last integral above is reduced to $\Om_{3\ep}\setminus \ov{\Om_{4\ep}}$ since the derivatives of $\chi_1$ vanish both on $\ov{\Om_{4\ep}}$ and in $\Om\setminus\ov{\Om_{3\ep}}$. In addition, we have $\vp(\cdot,t_0)=e^{\la(d+c_0)}< e^{\la(4\ep+c_0)}$ in $\Om_{3\ep}\setminus \ov{\Om_{4\ep}}$. Thus we combine the above two inequalities to obtain
\begin{equation}
\label{eq7n}
\begin{aligned} 
&\int_{\Om_{4\ep}} \big(|\na\nu|^2 + |\na\kappa|^2 + s^2\la^2\vp^2 (|\nu|^2 + |\kappa|^2)\big)e^{2s\vp} dx \le Ce^{2se^{\la(4\ep+c_0)}} (\|\nu\|_{L^2(\Om_{3\ep})}^2 + \|\kappa\|_{H^1(\Om_{3\ep})}^2) \\
&\hspace{2cm} +C\int_{\Om_{3\ep}} \Big( |\mathrm{div}(2\nu\mathcal{E}(u_1))|^2 + |\mathrm{rot}(\kappa\mathrm{rot}H_1)|^2 + \frac{1}{s^2\la^2\vp^2}|\na (\mathrm{rot}(\kappa\mathrm{rot}H_1))|^2 \Big) e^{2s\vp}dx 
\end{aligned}
\end{equation}
for all $\la\ge \la_0$, $s\ge s_0$ and $t=t_0$.

\begin{align}
\nonumber
&\text{[the second term on the RHS of \eqref{eq7n}]} \le C\int_{\Om_{3\ep}} (|\pa_t u|^2 + |\Delta u|^2 + |\na u|^2 + |L_1(H,\na H)|^2 +|\na p|^2 )e^{2s\vp}dx \\
\nonumber
&\hspace{2.5cm} + C\int_{\Om_{3\ep}} (|\pa_t H|^2 + |\Delta H|^2 + |\na H|^2 + |L_2(u,\na u)|^2 )e^{2s\vp}dx \\
\nonumber
&\hspace{2.5cm} + C\int_{\Om_{3\ep}} \frac{1}{s^2\la^2\vp^2}\big( |\na (\pa_t H)|^2 + |\na(\Delta H)|^2 + \sum_{i,j=1}^3|\pa_i\pa_j H|^2 + |\na (L_2(u,\na u))|^2\big)e^{2s\vp}dx \\
\label{eq8n}
&\hspace{2.8cm} \le C\int_{\Om_{3\ep}} (|\pa_t u|^2 + |\pa_t H|^2 + \frac{1}{s^2\la^2\vp^2}|\na (\pa_t H)|^2)e^{2s\vp}dx + C\mathcal{D}_1^2
\end{align}
where 
$$
\mathcal{D}_1:= \|u(\cdot,t_0)\|_{H^2(\Om_{3\ep})} + \|H(\cdot,t_0)\|_{H^3(\Om_{3\ep})} + \|\na p(\cdot,t_0)\|_{L^2(\Om_{3\ep})}.
$$
Next, we introduce another level sets:
\begin{align*}
Q_\ep := \{(x,t)\in Q: \psi(x,t) > \ep + c_0\}\qquad \text{for any }\ep>0.
\end{align*}
Then we have the following relations:
\vspace{0.1cm}

(\romannumeral1)  $Q_\ep\subset\Om_{\ep}\times (0,T)$, 
\vspace{0.1cm}

(\romannumeral2)  $Q_\ep\supset\Om_{\ep}\times \{t_0\}$.
\vspace{0.1cm}

\noindent In fact, if $(x,t)\in Q_\ep$, we have $d(x) - \be(t-t_0)^2 > \ep$, i.e. $d(x)>\be(t-t_0)^2 + \ep> \ep$. This means $x\in\Om_{\ep}$. (\romannumeral1) is verified. On the other hand, if $x\in\Om_\ep$ and $t=t_0$ then $\psi_2(x,t) = d(x) - \be(t-t_0)^2 + c_0 = d(x) + c_0 > \ep + c_0$. That is, $(x,t)\in Q_\ep$. (\romannumeral2) is verified.
Furthermore, we choose $\be = \f{\|d\|_{C(\ov{\Om_1})}}{\de^2}$ where $\de:= \min\{t_0, T- t_0\}$ so that 
\vspace{0.1cm}

(\romannumeral3)  $\ov{Q_\ep}\cap (\Om\times\{0,T\}) = \emptyset$
\vspace{0.1cm}

\noindent is valid. Indeed, for $\forall (x,t)\in \Om\times\{0,T\}$, $\psi(x,t) = d(x) - \be(t-t_0)^2 + c_0 \le \|d\|_{C(\ov{\Om_1})} - \be\de^2 + c_0 = c_0$. This leads to $(x,t)\notin \ov{Q_\ep}$. 

Relations (\romannumeral1) -- (\romannumeral3) guarantee that $Q_\ep$ is a sub-domain of $Q$ and $\pa Q_\ep\cap \pa Q\subset \Ga\times (0,T)$. Moreover, we assert that 
\vspace{0.1cm}

(\romannumeral4)   $\Om_{3\ep}\times (t_0-\delta_{\ep},t_0)\subset Q_{2\ep}, \quad \de_\ep := \sqrt{\f{\ep}{\be}} = \sqrt{\f{\ep}{\|d\|}}\de$.
\vspace{0.1cm}

\noindent Actually, for any $(x,t)\in \Om_{3\ep}\times (t_0-\delta_{\ep},t_0)$, we have
$$
\psi(x,t) = d(x) - \beta(t-t_0)^2 + c_0 > 3\ep - \beta\delta_{\ep}^2 + c_0 = 2\ep + c_0
$$
which implies $(x,t)\in Q_{2\ep}$. 

Now we construct a function $\eta\in C^2[0,T]$ such that $0\le \eta\le 1$ and
$$
\eta = \left\{
\begin{aligned}
& 1 \quad in\ [t_0-\f12\de_\ep, t_0+\f12\de_\ep], \\
& 0 \quad in\ [0,t_0-\de_\ep]\cup [t_0+ \de_\ep, T]
\end{aligned}
\right. 
$$
for any small $\ep<\|d\|_{C(\ov{\Om_1})}$. Then by noting that $\eta(t_0-\de_\ep) = 0$, $\eta(t_0) = 1$, we have
\begin{align}
\nonumber
&\int_{\Om_{3\ep}} |\pa_t u(x,t_0)|^2e^{2s\vp(x,t_0)}dx = \int_{t_0-\de_\ep}^{t_0} \f{\pa}{\pa t}\bigg( \eta \int_{\Om_{3\ep}} |\pa_t u|^2e^{2s\vp}dx\bigg) dt \\
\nonumber
&\hspace{3.7cm}= \int_{\Om_{3\ep}}\int_{t_0-\de_\ep}^{t_0} \bigg( 2\eta(\pa_t u\cdot\pa_t^2 u) + \big(2\eta s(\pa_t\vp) + \eta^\prime \big) |\pa_t u|^2 \bigg) e^{2s\vp}dxdt \\
\label{eq9n}
&\hspace{3.7cm}\le C(\la)s^{-1}\int_{Q_{2\ep}} (|\pa_t^2 u|^2 + s^2|\pa_t u|^2)e^{2s\vp}dxdt.
\end{align}
We used $s\ge 1$ and
\begin{align*}
&|\pa_t \vp| = 2\beta\la\vp |t-t_0| \le C(\la), \\
&2|\pa_t u\cdot\pa_t^2 u| \le \f{1}{s}|\pa_t^2 u|^2 + s|\pa_t u|^2.
\end{align*}
Similarly, we have
\begin{align}
\nonumber
&\int_{\Om_{3\ep}} |\pa_t H(x,t_0)|^2e^{2s\vp(x,t_0)}dx = \int_{t_0-\de_\ep}^{t_0} \f{\pa}{\pa t}\bigg( \eta \int_{\Om_{3\ep}} |\pa_t H|^2e^{2s\vp}dx\bigg) dt \\
\nonumber
&\hspace{3.7cm}= \int_{\Om_{3\ep}}\int_{t_0-\de_\ep}^{t_0} \bigg( 2\eta (\pa_t H\cdot\pa_t^2 H) + \big(2\eta s(\pa_t\vp) + \eta^\prime \big) |\pa_t H|^2 \bigg) e^{2s\vp}dxdt \\
\label{eq10n}
&\hspace{3.7cm}\le C(\la)s^{-1}\int_{Q_{2\ep}} (|\pa_t^2 H|^2 + s^2|\pa_t H|^2)e^{2s\vp}dxdt.
\end{align}
and
\begin{align}
\nonumber
&\int_{\Om_{3\ep}} \frac{1}{s^2\la^2\vp(x,t_0)^2}|\na(\pa_t H(x,t_0))|^2e^{2s\vp(x,t_0)}dx = \int_{t_0-\de_\ep}^{t_0} \f{\pa}{\pa t}\bigg( \eta \int_{\Om_{3\ep}} \frac{1}{s^2\la^2\vp^2}|\na(\pa_t H)|^2e^{2s\vp}dx\bigg) dt \\
\nonumber
&\hspace{1cm}= \int_{\Om_{3\ep}}\int_{t_0-\de_\ep}^{t_0} \bigg( \frac{2\eta}{s^2\la^2\vp^2}(\na(\pa_t H):\na(\pa_t^2 H)) + \frac{2\eta s(\pa_t\vp) + \eta^\prime}{s^2\la^2\vp^2} |\na(\pa_t H)|^2 \bigg) e^{2s\vp} dxdt \\
\label{eq11n}
&\hspace{1cm}\le C(\la)s^{-1}\int_{Q_{2\ep}} (s^{-2}|\na(\pa_t^2 H)|^2 + |\na(\pa_t H)|^2)e^{2s\vp} dxdt.
\end{align}

Set $w_1=\pa_t u$, $w_2=\pa_t^2 u$, $q_1=\pa_t p$, $q_2=\pa_t^2 p$ and $h_1=\pa_t H$, $h_2=\pa_t^2 H$. Furthermore we denote
$$
\mathcal{L}_1(u,p,H) := \pa_t u - \nu_2\Delta u + (u\!\cdot\!\nabla) u_2 + ((u_1-\!\nabla \nu_2)\!\cdot\!\nabla) u - \nabla u^T\!\cdot\!\nabla\nu_2 + L_1(H,\nabla H) + \na p
$$
and
$$
\mathcal{L}_2(u,H) := \pa_t H - \kappa_2\Delta H - (H\!\cdot\!\nabla) u_2 + (u_1\!\cdot\!\nabla) H + \nabla\kappa_2\times\mathrm{rot}H + L_2(u,\nabla u).
$$
Then according to our governing system (\ref{sy:MHD6}), we have
$$
\left\{
\begin{aligned}
&\ \mathcal{L}_1(u,p,H) = \mathrm{div}(2\nu\mathcal{E}(u_1)), \\
&\ \mathcal{L}_2(u,H)  = -\mathrm{rot}(\kappa\mathrm{rot}H_1), \\
&\ \mathrm{div}\; u = 0, \quad \mathrm{div}\; H = 0
\end{aligned}
\right.
$$
and
\vspace{0.2cm}
$$
\left\{
\begin{aligned}
&\mathcal{L}_1(w_1,q_1,h_1) = \mathrm{div}(2\nu\mathcal{E}((\pa_t u_1))) - (u\!\cdot\!\na)(\pa_t u_2) - ((\pa_t u_1)\!\cdot\!\nabla) u - L_{1t}(H,\na H), \\
&\mathcal{L}_2(w_1,h_1) = -\mathrm{rot}(\kappa\mathrm{rot}(\pa_t H_1)) + (H\!\cdot\!\nabla) (\pa_t u_2) - ((\pa_t u_1)\!\cdot\!\nabla) H - L_{2t}(u,\na u), \\
& \mathrm{div}\; w_1 = 0, \quad \mathrm{div}\; h_1 = 0
\end{aligned}
\right.
$$
and
\vspace{0.2cm}
$$
\left\{
\begin{aligned}
&\mathcal{L}_1(w_2,q_2,h_2) = \mathrm{div}(2\nu\mathcal{E}((\pa_t^2 u_1))) - 2(w_1\!\cdot\!\na)(\pa_t u_2) - 2((\pa_t u_1)\!\cdot\!\nabla) w_1 - 2L_{1t}(h_1,\na h_1) \\
&\hspace{5cm} - (u\!\cdot\!\na)(\pa_t^2 u_2) - ((\pa_t^2 u_1)\!\cdot\!\nabla) u - L_{1tt}(u,\na u), \\
&\mathcal{L}_2(w_2,h_2) = -\mathrm{rot}(\kappa\mathrm{rot}(\pa_t^2 H_1)) + 2(h_1\!\cdot\!\nabla) (\pa_t u_2) - 2((\pa_t u_1)\!\cdot\!\nabla) h_1 - 2L_{2t}(w_1,\na w_1) \\
&\hspace{5cm} + (H\!\cdot\!\nabla) (\pa_t^2 u_2) - ((\pa_t^2 u_1)\!\cdot\!\nabla) H - L_{2tt}(u,\na u), \\
& \mathrm{div}\; w_2 = 0, \quad \mathrm{div}\; h_2 = 0.
\end{aligned}
\right.
$$
By choosing a cut-off function $\chi_2\in C^\infty(\Bbb{R}^4)$ which satisfies $0\le\chi_2\le 1$ and
$$
\chi_2 =\left\{
\begin{aligned}
& 1\qquad \text{for }\psi > 2\ep + c_0, \\
& 0\qquad \text{for }\psi < \ep +c_0,
\end{aligned}
\right.
$$
we rewrite the above three systems
$$
\left\{
\begin{aligned}
&\ \mathcal{L}_1(\wt u,\wt p,\wt H) = \big(\mathcal{L}_1(\wt u,\wt p,\wt H) - \chi_2\mathcal{L}_1(u,p,H) \big) +\chi_2\mathrm{div}(2\nu\mathcal{E}(u_1)), \\
&\ \mathcal{L}_2(\wt u,\wt H)  = \big(\mathcal{L}_2(\wt u,\wt H) - \chi_2\mathcal{L}_2(u,H) \big)-\chi_2\mathrm{rot}(\kappa\mathrm{rot}H_1), \\
&\ \mathrm{div}\; \wt u = \na\chi_2\cdot u
\end{aligned}
\right.
$$
and
\vspace{0.2cm}
$$
\left\{
\begin{aligned}
&\mathcal{L}_1(\wt w_1,\wt q_1,\wt h_1) = \big(\mathcal{L}_1(\wt w_1,\wt q_1,\wt h_1) - \chi_2\mathcal{L}_1(w_1,q_1,h_1)\big) + \chi_2\mathrm{div}(2\nu\mathcal{E}((\pa_t u_1))) - \chi_2(u\!\cdot\!\na)(\pa_t u_2) \\
&\hspace{2cm} - \chi_2((\pa_t u_1)\!\cdot\!\nabla) u - \chi_2 L_{1t}(H,\na H), \\
&\mathcal{L}_2(\wt w_1,\wt h_1) = \big(\mathcal{L}_2(\wt w_1,\wt h_1) - \chi_2\mathcal{L}_2(w_1,h_1)\big) - \chi_2\mathrm{rot}(\kappa\mathrm{rot}(\pa_t H_1)) + \chi_2(H\!\cdot\!\nabla) (\pa_t u_2) \\
&\hspace{2cm} - \chi_2((\pa_t u_1)\!\cdot\!\nabla) H - \chi_2 L_{2t}(u,\na u), \\
& \mathrm{div}\; \wt w_1 = \na\chi_2\cdot w_1
\end{aligned}
\right.
$$
and
\vspace{0.2cm}
$$
\left\{
\begin{aligned}
&\mathcal{L}_1(\wt w_2,\wt q_2,\wt h_2) = \big(\mathcal{L}_1(\wt w_2,\wt q_2,\wt h_2) -\chi_2\mathcal{L}_1(w_2,q_2,h_2)\big) + \chi_2\mathrm{div}(2\nu\mathcal{E}((\pa_t^2 u_1))) - 2\chi_2(w_1\!\cdot\!\na)(\pa_t u_2) \\
&\hspace{0.8cm} - 2\chi_2((\pa_t u_1)\!\cdot\!\nabla) w_1 - 2\chi_2 L_{1t}(h_1,\na h_1) - \chi_2(u\!\cdot\!\na)(\pa_t^2 u_2) - \chi_2((\pa_t^2 u_1)\!\cdot\!\nabla) u - \chi_2 L_{1tt}(u,\na u), \\
&\mathcal{L}_2(\wt w_2,\wt h_2) = \big(\mathcal{L}_2(\wt w_2,\wt h_2) - \chi_2\mathcal{L}_2(w_2,h_2)\big) - \chi_2\mathrm{rot}(\kappa\mathrm{rot}(\pa_t^2 H_1)) + 2\chi_2(h_1\!\cdot\!\nabla) (\pa_t u_2) \\
&\hspace{0.8cm} - 2\chi_2((\pa_t u_1)\!\cdot\!\nabla) h_1 - 2\chi_2 L_{2t}(w_1,\na w_1) + \chi_2(H\!\cdot\!\nabla) (\pa_t^2 u_2) - \chi_2((\pa_t^2 u_1)\!\cdot\!\nabla) H - \chi_2 L_{2tt}(u,\na u), \\
& \mathrm{div}\; \wt w_2 = \na\chi_2\cdot w_2
\end{aligned}
\right.
$$
where $\wt u = \chi_2 u$, $\wt w_1 =\chi_2 w_1$, $\wt w_2 = \chi_2 w_2$, etc.

Then we can employ Carleman estimate (Theorem \ref{thm:CEDPn}) to $(\wt u,\wt p,\wt H)$, $(\wt w_1,\wt q_1,\wt h_1)$ and $(\wt w_2,\wt q_2,\wt h_2)$ respectively and obtain
\begin{equation*}
\begin{aligned}
\|&(\wt u,\wt p,\wt H)\|_{\sigma_s(Q)}^2 \le \;C\int_Q s\vp\chi_2^2(|\nu|^2 + |\na \nu|^2 + |\kappa|^2 + |\na \kappa|^2 )e^{2s\vp}dxdt \\
&\hspace{0cm} + C\int_Q \Big(\sum_{i,j=1}^3|\pa_i\pa_j\chi_2|^2 + |\na_{x,t}\chi_2|^2 + |\na(\pa_t \chi_2)|^2\Big)(s\vp(|\na_{x,t} u|^2 + |u|^2) + |\na H|^2 + |H|^2 + |p|^2)e^{2s\vp}dxdt \\
&\hspace{0cm} + Ce^{Cs}\bigg(\|\wt u\|_{L^2(\Sigma)}^2 + \|\na_{x,t} \wt u\|_{L^2(\Sigma)}^2 + \|\wt H\|_{L^2(\Sigma)}^2 + \|\na_{x,t} \wt H\|_{L^2(\Sigma)}^2 + \|\wt p\|_{L^2(0,T;H^{\frac{1}{2}}(\pa \Om))}^2 \bigg)
\end{aligned}
\end{equation*}
and
\begin{equation*}
\begin{aligned}
&\|(\wt w_1,\wt q_1,\wt h_1)\|_{\sigma_s(Q)}^2 \le \;C\int_Q s\vp\chi_2^2(|\nu|^2 + |\na \nu|^2 + |\kappa|^2 + |\na \kappa|^2 + |u|^2 + |\na u|^2 + |H|^2 + |\na H|^2)e^{2s\vp}dxdt \\
&\hspace{0cm} + C\int_Q \Big(\sum_{i,j=1}^3|\pa_i\pa_j\chi_2|^2 + |\na_{x,t}\chi_2|^2 + |\na(\pa_t \chi_2)|^2\Big)(s\vp(|\na_{x,t} w_1|^2 + |w_1|^2) + |\na h_1|^2 + |h_1|^2 + |q_1|^2)e^{2s\vp}dxdt \\
&\hspace{0cm} + Ce^{Cs}\bigg(\|\wt w_1\|_{L^2(\Sigma)}^2 + \|\na_{x,t} \wt w_1\|_{L^2(\Sigma)}^2 + \|\wt h_1\|_{L^2(\Sigma)}^2 + \|\na_{x,t} \wt h_1\|_{L^2(\Sigma)}^2 + \|\wt q_1\|_{L^2(0,T;H^{\frac{1}{2}}(\pa \Om))}^2 \bigg)
\end{aligned}
\end{equation*}
and
\begin{equation*}
\begin{aligned}
&\|(\wt w_2,\wt q_2,\wt h_2)\|_{\sigma_s(Q)}^2 \le \;C\int_Q s\vp\chi_2^2(|\nu|^2 + |\na \nu|^2 + |\kappa|^2 + |\na \kappa|^2 + |u|^2 + |\na u|^2 )e^{2s\vp}dxdt \\
&\hspace{0cm} + C\int_Q s\vp\chi_2^2(|H|^2 + |\na H|^2 + |w_1|^2 + |\na w_1|^2 + |h_1|^2 + |\na h_1|^2)e^{2s\vp}dxdt \\
&\hspace{0cm} + C\int_Q \Big(\sum_{i,j=1}^3|\pa_i\pa_j\chi_2|^2 + |\na_{x,t}\chi_2|^2 + |\na(\pa_t \chi_2)|^2\Big)(s\vp(|\na_{x,t} w_2|^2 + |w_2|^2) + |\na h_2|^2 + |h_2|^2 + |q_2|^2)e^{2s\vp}dxdt \\
&\hspace{0cm} + Ce^{Cs}\bigg(\|\wt w_2\|_{L^2(\Sigma)}^2 + \|\na_{x,t} \wt w_2\|_{L^2(\Sigma)}^2 + \|\wt h_2\|_{L^2(\Sigma)}^2 + \|\na_{x,t} \wt h_2\|_{L^2(\Sigma)}^2 + \|\wt q_2\|_{L^2(0,T;H^{\frac{1}{2}}(\pa \Om))}^2 \bigg)
\end{aligned}
\end{equation*}
for all large fixed $\la$ and all $s\ge \hat{s}$. 
Combining the above three estimates and absorb the lower-order terms on the RHS which leads to
\begin{equation}
\label{eq12n}
\begin{aligned}
\sum_{j=0}^2\|(\pa_t^j u,\pa_t^j p, \pa_t^j H)\|_{\sigma_s(Q_{2\ep})}^2 \le\; C\int_{Q_\ep} s\vp(|\nu|^2 + |\na \nu|^2 + |\kappa|^2 + |\na \kappa|^2)e^{2s\vp}dxdt + Low + Ce^{Cs}\mathcal{D}_2^2
\end{aligned}
\end{equation}
for all large fixed $\la\ge \la_0$ and all $s\ge \hat{s}$. Here
$$
\begin{aligned}
&Low := C\int_{Q_\ep} \Big(\sum_{i,j=1}^3|\pa_i\pa_j\chi_2|^2 + |\na_{x,t}\chi_2|^2 + |\na(\pa_t \chi_2)|^2\Big)(s\vp(|\na_{x,t} u|^2 + |u|^2) + |\na H|^2 + |H|^2 + |p|^2)e^{2s\vp}dxdt \\
&\hspace{0cm} +C\int_{Q_\ep} \Big(\sum_{i,j=1}^3|\pa_i\pa_j\chi_2|^2 + |\na_{x,t}\chi_2|^2 + |\na(\pa_t \chi_2)|^2\Big)(s\vp(|\na_{x,t} w_1|^2 + |w_1|^2) + |\na h_1|^2 + |h_1|^2 + |q_1|^2)e^{2s\vp}dxdt \\
&\hspace{0cm} + C\int_{Q_\ep} \Big(\sum_{i,j=1}^3|\pa_i\pa_j\chi_2|^2 + |\na_{x,t}\chi_2|^2 + |\na(\pa_t \chi_2)|^2\Big)(s\vp(|\na_{x,t} w_2|^2 + |w_2|^2) + |\na h_2|^2 + |h_2|^2 + |q_2|^2)e^{2s\vp}dxdt \\
&\mathcal{D}_2 := \sum_{j=0}^2 \Big(\|\pa_t^j u\|_{L^2(\Gamma\times (0,T))} + \|\pa_t^j(\na_{x,t}u)\|_{L^2(\Gamma\times (0,T))}\Big) + \sum_{j=0}^2 \Big(\|\pa_t^j H\|_{L^2(\Gamma\times (0,T))} + \|\pa_t^j(\na_{x,t}H)\|_{L^2(\Gamma\times (0,T))}\Big) \\
&\hspace{1cm}  + \sum_{j=0}^2 \Big(\|\pa_t^j p\|_{L^2(0,T;H^{\frac{1}{2}}(\Gamma))}\Big). 
\end{aligned}
$$
From the choice of $\chi_2$, we see that the derivatives of it vanishes in $Q_{2\ep}$. Since $\psi(x,t)$ has an upper bound $2\ep + c_0$ when $(x,t)$ is outside of $Q_{2\ep}$, we can simplify $Low$:
$$
Low \le Cse^{2se^{\la(2\ep+c_0)}}\sum_{j=0}^2\Big(\|\pa_t^j u\|_{H^{1,1}(Q_\ep)}^2 + \|\pa_t^j H\|_{H^{1,0}(Q_\ep)}^2 + \|\pa_t^j p\|_{L^2(Q_\ep)}^2\Big) =: Cse^{2se^{\la(2\ep+c_0)}} M_1^2
$$
where $M$ is defined in Theorem \ref{thm:stabilityn}. In terms of \eqref{eq12n}, immediately we have
\begin{align*}
&\int_{Q_{2\ep}} \big(s^2|\pa_t u|^2 + s^2|\pa_t^2 u|^2 + s^2|\pa_t H|^2 + s^2|\pa_t^2 H|^2 + |\na (\pa_t H)|^2) + |\na(\pa_t^2 H)|^2\big) e^{2s\vp}dxdt \\
&\hspace{2cm} \le C\int_{Q_\ep} (|\nu|^2 + |\na \nu|^2 + |\kappa|^2 + |\na \kappa|^2)e^{2s\vp}dxdt + Ce^{2se^{\la(2\ep+c_0)}} M_1^2 + Ce^{Cs}\mathcal{D}_2^2.
\end{align*}
We thus insert above inequality to \eqref{eq9n}-\eqref{eq11n} and obtain 
$$
\begin{aligned}
&C\int_{\Om_{3\ep}} (|\pa_t u(x,t_0)|^2 + |\pa_t H(x,t_0)|^2 + \frac{1}{s^2\la^2\vp^2(x,t_0)}|\na(\pa_t H(x,t_0))|^2)e^{2s\vp(x,t_0)}dx \\
&\hspace{1cm} \le Cs^{-1}\int_{Q_{\ep}} (|\nu|^2 + |\na \nu|^2 + |\kappa|^2 + |\na \kappa|^2)e^{2s\vp}dxdt + Ce^{2se^{\la(2\ep+c_0)}} M_1^2 + Ce^{Cs}\mathcal{D}_2^2
\end{aligned}
$$
which along with \eqref{eq7n} and \eqref{eq8n} yields
\begin{align}
\nonumber
&\int_{\Om_{4\ep}} (|\na \nu|^2 + s^2\vp^2 |\nu|^2 + |\na \kappa|^2 + s^2\vp^2 |\kappa|^2)e^{2s\vp(x,t_0)} dx \le Cs^{-1}\int_{Q_\ep} (|\nu|^2 + |\na \nu|^2 + |\kappa|^2 + |\na \kappa|^2)e^{2s\vp}dxdt \\
\label{eq13n}
&\hspace{2cm} + Ce^{2se^{\la(4\ep+c_0)}}(\|\nu\|_{L^2(\Om_{3\ep})}^2 + \|\kappa\|_{H^1(\Om_{3\ep})}^2) + Ce^{2se^{\la(2\ep+c_0)}} M_1^2 + Ce^{Cs}\mathcal{D}^2
\end{align} 
where
\begin{align*}
\mathcal{D}^2 = \mathcal{D}_1^2 + \mathcal{D}_2^2.
\end{align*}
We carefully calculate the first term on the RHS of \eqref{eq13n}: 
\begin{align*}
&Cs^{-1}\int_{Q_\ep} (|\nu|^2 + |\na \nu|^2 + |\kappa|^2 + |\na \kappa|^2)e^{2s\vp}dxdt \\
&= Cs^{-1}\int_{Q_{4\ep}} (|\nu|^2 + |\na \nu|^2 + |\kappa|^2 + |\na \kappa|^2)e^{2s\vp}dxdt + Cs^{-1}\int_{Q_\ep\setminus Q_{4\ep}} (|\nu|^2 + |\na \nu|^2 + |\kappa|^2 + |\na \kappa|^2)e^{2s\vp}dxdt \\
&\le Cs^{-1}\int_0^T\int_{\Om_{4\ep}} (|\nu|^2 + |\na \nu|^2 + |\kappa|^2 + |\na \kappa|^2)e^{2s\vp}dxdt + CTe^{2se^{\la(4\ep+c_0)}}\int_{\Om_{\ep}} (|\nu|^2 + |\na \nu|^2 + |\kappa|^2 + |\na \kappa|^2)dx \\
&\le Cs^{-1} \int_{\Om_{4\ep}} (|\nu|^2 + |\na \nu|^2 + |\kappa|^2 + |\na \kappa|^2)e^{2s\vp(x,t_0)}dx + Ce^{2se^{\la(4\ep+c_0)}}(\|\nu\|_{H^1(\Om_{\ep})}^2 + \|\kappa\|_{H^1(\Om_{\ep})}^2)
\end{align*}
where we note that $\vp(x,t)$ attains its maximum at $t=t_0$ for any $x\in \Om$. Thus we can absorb the first term on the RHS above by taking $s$ large (e.g. $s\ge s_1$) which gives
\begin{align}
\label{eq14n}
\int_{\Om_{4\ep}} (|\na \nu|^2 + s^2\vp^2 |\nu|^2 + |\na \kappa|^2 + s^2\vp^2 |\kappa|^2)e^{2s\vp(x,t_0)} dx \le Ce^{2se^{\la(4\ep+c_0)}} M^2 + Ce^{Cs}\mathcal{D}^2.
\end{align}
Here 
$$
M^2 = M_1^2 + \|\nu\|_{H^1(\Om_{3\ep})}^2 + \|\kappa\|_{H^1(\Om_{3\ep})}^2.
$$
On the other hand, the LHS of \eqref{eq14n} can be estimated from below:
\begin{align*}
&\int_{\Om_{4\ep}} (|\na \nu|^2 + s^2\vp^2 |\nu|^2 + |\na \kappa|^2 + s^2\vp^2 |\kappa|^2)e^{2s\vp(x,t_0)} dx 
\ge \int_{\Om_{5\ep}} (|\na \nu|^2 + |\nu|^2 + |\na \kappa|^2 + |\kappa|^2)e^{2s\vp(x,t_0)} dx \\
&\hspace{8.2cm} \ge e^{2se^{\la(5\ep+c_0)}}(\|\nu\|_{H^1(\Om_{5\ep})}^2 + \|\kappa\|_{H^1(\Om_{5\ep})}^2)
\end{align*}
Therefore \eqref{eq14n} indicates
\begin{align}
\label{eq15n}
&\|\nu\|_{H^1(\Om_{5\ep})}^2 + \|\kappa\|_{H^1(\Om_{5\ep})}^2 \le Ce^{-C_0s}M^2 + Ce^{Cs}\mathcal{D}^2
\end{align}
for all $s\ge s_2=\max\{s_0,s_1\}$ with $C_0 = 2e^{\la(4\ep+c_0)}(e^{\la\ep} -1)>0$. We can substitute $s$ by $s+s_2$ so that \eqref{eq15n} holds for all $s\ge 0$. 

Finally, we apply a well-known argument to reach the stability inequality of H\"{o}lder type \eqref{eq:loc-stab}. For reference, see the final step of the proof on pp.28 in \cite{Y09}. This completes the proof of our main result.

\qed

\noindent $\mathbf{Remark.}$ Sometimes the following case is considered. The coefficients $\nu,\kappa$ are given in a more general form of 
$$
\nu(x,t) = \tilde{\nu}(x)r_1(x,t), \quad \kappa(x,t) = \tilde{\kappa}(x)r_2(x,t),\qquad for\ (x,t)\in Q
$$
provided $r_1,r_2$ are two given functions. The above stability inequality for $\tilde{\nu}$ and $\tilde{\kappa}$ still holds if we add some smoothness and nonzero assumptions to $r_1$ and $r_2$. The proof is similar but it is necessary to pay more attention to the order of large parameter $s$.


\end{document}